\documentclass[11pt]{amsart}

\usepackage{amscd,latexsym,amsthm,amsfonts,amssymb,amsmath,amsxtra,color}
\usepackage[mathscr]{eucal}
\usepackage{pictexwd,dcpic}
\usepackage[normalem]{ulem}
\usepackage{hyperref}
\usepackage{enumitem}

\renewcommand\theequation{\thesection.\arabic{equation}}

\newcommand{\BC}{{\mathbb {C}}}

\newcommand{\BR}{{\mathbb {R}}}

\newcommand{\BZ}{{\mathbb {Z}}}

\newcommand{\CB}{{\mathcal {B}}}
\newcommand{\CC}{{\mathcal {C}}}

\newcommand{\CF}{{\mathcal {F}}}

\newcommand{\CK}{{\mathcal {K}}}
\newcommand{\CL}{{\mathcal {L}}}

\newcommand{\CO}{{\mathcal {O}}}
\newcommand{\CP}{{\mathcal {P}}}

\newcommand{\CT}{{\mathcal {T}}}

\newcommand{\CX}{{\mathcal {X}}}

\newcommand{\Fa}{{\mathfrak {a}}}

\newcommand{\Fg}{{\mathfrak {g}}}
\newcommand{\Fh}{{\mathfrak {h}}}

\newcommand{\Fm}{{\mathfrak {m}}}

\newcommand{\Ft}{{\mathfrak {t}}}
\newcommand{\Fu}{{\mathfrak {u}}}

\newcommand{\Fz}{{\mathfrak {z}}}

\newcommand{\RU}{{\mathrm {U}}}

\newcommand{\Ad}{{\mathrm{Ad}}}

\newcommand{\End}{{\mathrm{End}}}

\newcommand{\GL}{{\mathrm{GL}}}

\newcommand{\GU}{{\mathrm{GU}}}
\newcommand{\Hom}{{\mathrm{Hom}}}

\renewcommand{\Im}{{\mathrm{Im}}}

\newcommand{\I}{{\mathrm{I}}}

\newcommand{\tr}{{\mathrm{tr}}}

\newcommand{\ud}{\,\mathrm{d}}

\newcommand{\vol}{{\mathrm{vol}}}

\newcommand{\zg}{Z_G(F)\backslash G(F)}
\newcommand{\zh}{Z_H(F)\backslash H(F)}
\newcommand{\back}{\backslash}

\newcommand{\hc}{\Xi^{H\backslash G}}
\newcommand{\nor}{\sigma_{H\backslash G}}

\def\lam{{\lambda}}

\newtheorem{thm}{Theorem}[section]
\newtheorem{cor}[thm]{Corollary}
\newtheorem{lem}[thm]{Lemma}

\newtheorem{prop}[thm]{Proposition}
\newtheorem {conj}[thm]{Conjecture}

\newtheorem {ques/conj}[thm]{Question/Conjecture}

\newtheorem{defn}[thm]{Definition}
\newtheorem{rmk}[thm]{Remark}

\makeatletter

\newcommand{\Rmnum}[1]{\expandafter\@slowromancap\romannumeral #1@}
\makeatother

\begin{document}
\renewcommand{\theequation}{\arabic{equation}}
\numberwithin{equation}{section}

\title[]{The multiplicity problems for the unitary Ginzburg-Rallis models}

\author{Chen Wan}
\address{Department of Mathematics\\
Massachusetts Institute of Technology\\
Cambridge, USA}
\email{chenwan@mit.edu}

\author{Lei Zhang}
\address{Department of Mathematics\\
National University of Singapore, Singapore}
\email{matzhlei@nus.edu.sg}


\begin{abstract}
We consider the local multiplicity problems of the analogy of the Ginzburg-Rallis model for the unitary group and the unitary similitude  group cases. For the unitary similitude  group case, by proving a local trace formula for the model, we are able to prove a multiplicity formula for all tempered representations, which implies that the summation of the multiplicities is equal to $1$ over every tempered local Vogan $L$-packet. For the unitary group case, we also prove a multiplicity formula for all tempered representations which implies that the summation of the multiplicities is equal to $2$ over every tempered local Vogan $L$-packet.
\end{abstract}

\thanks{The work of the second named author is supported in part by AcRF Tier 1 grant R-146-000-237-114 of National University of Singapore.}

\maketitle

\section{Introduction and Main Results} \label{sec:introduction}
\subsection{Main results}
Let $F$ be a nonarchimedean field of characteristic 0 and $E=F(\sqrt{\alpha})$ be a quadratic extension of $F$. Let $\eta_{E/F}\colon  F^{\times}\rightarrow \BC^{\times}$ be the quadratic character associated to $E$ via the local class field theory, $N_{E/F}$ (resp. $\tr_{E/F}$) be the norm map (resp. trace map), and $x\rightarrow \bar{x}$ be the Galois action on $E$. Denote  $w_{n}$ to  be the symmetric matrix of size $n\times n$ given by
$$
w_{n}=\begin{pmatrix}
&w_{n-1}\\ 1&
\end{pmatrix}\text{ and } w_1=\begin{pmatrix}1 \end{pmatrix}.
$$
For $\varepsilon\in F^{\times}$, let
$$
J_{2n, \varepsilon}=\begin{pmatrix}0&0& w_{n-1}\\0&A_{\varepsilon}&0\\w_{n-1}&0&0\end{pmatrix}
\text{ where }A_{\varepsilon}=\begin{pmatrix}-\varepsilon&0\\0&1\end{pmatrix}.
$$
Define the unitary similitude group $\GU_{2n,\varepsilon}(F)=\GU(J_{2n,\varepsilon})(F)$ to be
\begin{equation}\label{eq:GU-n}
\GU(J_{2n,\varepsilon})(F)=\{ g\in \GL_{2n}(E)\colon  {}^{t}\bar{g}J_{2n,\varepsilon}g=\lam(g) J_{2n,\varepsilon}\}	
\end{equation}
where $\lam(g)\in F^{\times}$ is the similitude factor of $g$. Note that if $\varepsilon$ belongs to the image $ \Im(N_{E/F})$, then $\GU(J_{2n,\varepsilon})$ is quasi-split;
if $\varepsilon\notin \Im(N_{E/F})$, then $\GU(J_{2n,\varepsilon})$ is the non-quasi-split inner form of the quasi-split unitary similitude group.
In this paper, we mainly work on the groups $G_{\varepsilon}=\GU(J_{6,\varepsilon})$.

Next, we will introduce a spherical subgroup of $G_{\varepsilon}$. Let $P_{\varepsilon}=M_{\varepsilon}U_{\varepsilon}$ be the standard parabolic subgroup of $\GU(J_{6,\varepsilon})$ with
\begin{align*}
M_{\varepsilon}(F)=&\{m(g,h)=\left(\begin{smallmatrix}
g&&\\&h&\\&&\lambda(h)g^*	\end{smallmatrix}\right) \colon g\in \GL_2(E),\; g^*=w_2{}^t\bar{g}^{-1}w_2,~h\in\GU(J_{2,\varepsilon})(F)\},
\end{align*}
\begin{align*}
U_{\varepsilon}(F)=\{u(X,Y)=\left(\begin{smallmatrix}
I_2&X&Y\\&I_{2}&X'\\&&I_2	\end{smallmatrix}\right) \colon & X,Y\in Mat_{2\times 2}(E),X'=-A_{\varepsilon}^{-1}{}^t\!Xw_2,\\
&w_2Y + {}^t\!Yw_2+{}^t\!X'A_{\varepsilon}X'=0\}.
\end{align*}


Define a generic character $\xi_{\varepsilon}$ of $U_{\varepsilon}(F)$ to be
$$
\xi_{\varepsilon}(u(X,Y))=\psi(\tr_{E/F}(\tr(X))).
$$
where $\psi$ is a non-trivial additive character of $F$. Then the stabilizer of $\xi_{\varepsilon}$ under the adjoint action of $M_{\varepsilon}(F)$ is
$$
H_{0,\varepsilon}(F):=\{m(h,h)\colon h\in \GU(J_{2,\varepsilon})(F)\}.
$$
Let $\chi_F$ (resp. $\chi_E$) be a character of $F^{\times}$ (resp. $E^{\times}$). We then define the character $\omega_{\varepsilon}$ of $H_{0,\varepsilon}(F)$ to be
$$\omega_{\varepsilon}(m(h,h))=\chi_E(\det(h))\chi_F(\lambda(h))$$
where $\lambda$ is the similitude character of $\GU(J_{2,\varepsilon})(F)$. Let $\eta$ be the restriction of the character $\omega_{\varepsilon}$ to the center $Z_{H_{0,\varepsilon}}(F)=Z_{G_{\varepsilon}}(F)\simeq E^{\times}$. It is easy to see that $\eta=\chi_{E}^{2}\otimes (\chi_F\circ N_{E/F})$.

Define $H_\varepsilon=H_{0,\varepsilon}\ltimes U_{\varepsilon}$, which is a spherical subgroup of $\GU(J_{6,\varepsilon})$. Then we have a character $\omega_{\varepsilon}\otimes \xi_{\varepsilon}$ of $H_{\varepsilon}(F)$. Let $\pi_{\varepsilon}$ be a smooth admissible representation of $G_{\varepsilon}(F)$ with central character $\eta$. We define the multiplicity
$$m(\pi_{\varepsilon})=\dim(\Hom_{H_{\varepsilon}(F)} (\pi_{\varepsilon},\omega_{\varepsilon}\otimes \xi_{\varepsilon})).$$
The goal of this paper is to study the behavior of the multiplicity $m(\pi_{\varepsilon})$ over the local Vogan $L$-packet.

For $i=1,2$, fix $\varepsilon_i\in F^{\times}$ with $\eta_{E/F}(\varepsilon_i)=(-1)^{i-1}$. Let $\phi$ be a tempered Langlands parameter for $\GU_6(F)$. Assume the endoscopic classification holds for even unitary similitude group (This is expected from the endoscopic classification of unitary groups in \cite{M15} and \cite{KMSW}, together with Xu's work \cite{Xu16} on the reduction from the similitude classical groups to classical groups.
We refer the readers to Section \ref{sec:U-GU} for details). Then the parameter $\phi$ determines a tempered local Vogan $L$-packet $\Pi_{\phi}=\Pi_{\phi} (G_{\varepsilon_1})\cup \Pi_{\phi}(G_{\varepsilon_2})$ consisting of a finite number of tempered representations of $G_{\varepsilon_1}(F)$ and $G_{\varepsilon_2}(F)$ respectively. Our main theorem can be stated as follows.

\begin{thm}\label{main GU}
For all tempered Langlands parameters $\phi$ of $\GU_6(F)$, we have
$$
\sum_{i=1}^{2}\sum_{\pi_{\varepsilon_i}\in \Pi_{\phi} (G_{\varepsilon_i})} m(\pi_{\varepsilon_i})=1.
$$
In other words, the summation of the multiplicities over every tempered local Vogan $L$-packet is equal to 1.
\end{thm}

Then we study the analogy of the pair $(G_{\varepsilon},H_{\varepsilon})$ for the unitary group case. For $\varepsilon\in F^{\times}$, we define the unitary group $\mathrm{U}(J_{2n,\varepsilon})$ to be
\begin{equation}\label{eq:U-n}
\mathrm{U}(J_{2n,\varepsilon})(F)=\{ g\in \GL_{2n}(E)\colon {}^{t}\bar{g}J_{2n,\varepsilon}g=J_{2n,\varepsilon}\}.	
\end{equation}
We define $G_{1,\varepsilon}(F)=\mathrm{U}(J_{6,\varepsilon})(F)$. As in the similitude case, we can define the subgroups $H_{1,\varepsilon}= H_{0,1,\varepsilon}\ltimes U_{1,\varepsilon}$ of $G_{1,\varepsilon}$ with $H_{0,1,\varepsilon}(F)\simeq \mathrm{U}(J_{2,\varepsilon})(F)$. We can also define character $\omega_{1,\varepsilon}\otimes \xi_{1,\varepsilon}$ of $H_{1,\varepsilon}(F)$ via the characters $\psi$ and $\chi_E$ (note that here we don't have similitude character, hence we can only define the character $\omega_{1,\varepsilon}$ via the determinant map). Let $\eta_1$ be the restriction of the character $\omega_{1,\varepsilon}$ on the center $Z_{H_{0,1,\varepsilon}}(F)=Z_{G_{1,\varepsilon}}(F)\simeq E^{1}$
where $E^1$ is the kernel of the norm map $N_{E/F}$. It is easy to see that $\eta_1=\chi_{E}^{2}|_{E^1}$. Let $\pi_{1,\varepsilon}$ be a smooth admissible representation of $G_{1,\varepsilon}(F)$ with central character $\eta_1$, we define the multiplicity
$$m(\pi_{1,\varepsilon})=\dim(\Hom_{H_{1,\varepsilon}(F)} (\pi_{1,\varepsilon},\omega_{1,\varepsilon}\otimes \xi_{1,\varepsilon})).$$

For $i=1,2$, let $\varepsilon_i\in F^{\times}$ with $\eta_{E/F}(\varepsilon_i)=(-1)^{i-1}$ as before. Let $\phi$ be a tempered Langlands parameter for $\RU_6(F)$. By the endoscopic classification of unitary groups in \cite{M15} and \cite{KMSW}, the parameter $\phi$ determines a tempered local Vogan $L$-packet $\Pi_{\phi}=\Pi_{\phi} (G_{1,\varepsilon_1})\cup \Pi_{\phi}(G_{1,\varepsilon_2})$ consisting of a finite number of tempered representations of $G_{1,\varepsilon_1}(F)$ and $G_{1,\varepsilon_2}(F)$ respectively. Our main theorem for the unitary group case can be stated as follows.

\begin{thm}\label{main U}
For all tempered Langlands parameters $\phi$ of $\RU_6(F)$, we have
$$
\sum_{i=1}^{2}\sum_{\pi_{1,\varepsilon_i}\in \Pi_{\phi} (G_{1,\varepsilon_i})} m(\pi_{1,\varepsilon_i})=2.$$
In other words, the summation of the multiplicities over every tempered local Vogan $L$-packet is equal to 2.
\end{thm}

\begin{rmk}
The models $(G_{\varepsilon},H_{\varepsilon})$ (resp. $(G_{1,\varepsilon},H_{1,\varepsilon})$) can be viewed as the analogy of the Ginzburg-Rallis model (GR for simplicity) for the  unitary similitude group (resp. unitary group) case. The local multiplicity problem for the Ginzburg-Rallis model has been considered by the first named author in \cite{Wan15}, \cite{Wan16} and \cite{Wan16b}. We refer the readers to \cite{Wan17} for the definition of the model and the results.
\end{rmk}

\begin{rmk}\label{Gelfand}
We expect the results in Theorems \ref{main GU} and \ref{main U} hold for all generic local Vogan $L$-packets. For the  unitary similitude group case, we also expect the model $(G_{\varepsilon},H_{\varepsilon})$ to be a Gelfand pair, i.e. $m(\pi_{\varepsilon})\leq 1$ for all irreducible smooth representations of $G_{\varepsilon}(F)$.
Theorem \ref{main GU} verifies this inequality for all tempered representations.
However, the model $(G_{1,\varepsilon},H_{1,\varepsilon})$ is not a Gelfand pair. In fact, later in our proof, we can show that when $G_{1,\varepsilon}$ is quasi-split, $m(\pi_{1,\varepsilon})=2$ for all generic tempered unramified representations of $G_{1,\varepsilon}(F)$. For this model, we expect that the multiplicity is always less or equal to $2$.
\end{rmk}

\begin{rmk}
As in the Ginzburg-Rallis model case, globally we expect that the period integrals of the models $(G_{\varepsilon},H_{\varepsilon})$ are related to the central value of the exterior cube $L$-function.
\end{rmk}

\subsection{Remarks on the proofs}
We first discuss the proof of Theorem \ref{main GU} (i.e. the  unitary similitude group case). Our proof of Theorem \ref{main GU} uses Waldspurger's method in his proof of the local orthogonal Gan-Gross-Prasad (GGP for simplicity) conjecture in \cite{W10} and \cite{W12}. In other words, we are going to prove a multiplicity formula
$$m(\pi_{\varepsilon})=m_{geom}(\pi_{\varepsilon})$$
for all the tempered representations $\pi_{\varepsilon}$ of $G_{\varepsilon}(F)$. Here $m_{geom}(\pi_{\varepsilon})$ is defined in terms of the regular germs of the distribution character $\theta_{\pi_{\varepsilon}}$. Then Theorem \ref{main GU} will follow from the multiplicity formula together with the behavior of the distribution characters on the local $L$-packet. It is worth to mention that Waldspurger's method later has been adapted by Beuzart-Plessis in his proof of the local unitary GGP conjecture (\cite{B12}, \cite{B15}), and by the first named author in his work of the local Ginzburg-Rallis model (\cite{Wan15}, \cite{Wan16}). It has also been used in \cite{B17} and \cite{BW18} for the local multiplicity problems of the Galois model and the generalized Shalika model (but with different proofs of the geometric side of the trace formula).

In order to prove the multiplicity formula, as in all the previous cases, one needs to prove a local trace formula for the model. We refer the readers to Section \ref{sec:trace-formula} for the definitions of the trace formula and the multiplicity formula for the model $(G_{\varepsilon},H_{\varepsilon})$. Our proof for the geometric side of the trace formula is quite similar to the GGP case in \cite{B15} and all the computations are very similar to the GR case in \cite{Wan15}. As a result, we will only give a sketch of the proof without providing details (see Section \ref{sec:proof-geometric}).

As for the spectral side of the trace formula, our proof is quite different from the GGP case and the GR case. The main reason is that unlike the previous cases, we don't have the Gelfand pair condition for the model $(G_{\varepsilon},H_{\varepsilon})$ (although it is expected, see Remark \ref{Gelfand}). To avoid using the Gelfand pair condition, we decompose the Harish-Chandra-Schwartz space $\CC(G(F))$ into two subspaces $\CC(G(F))={}^{\circ}\CC(G(F))\oplus \CC_{ind}(G(F))$ where ${}^{\circ}\CC(G(F))$ corresponds to the discrete series and $\CC_{ind}(G(F))$ corresponds to the induced representations. Then we only need to prove the spectral expansions for these two subspaces. For the subspace ${}^{\circ}\CC(G(F))$, we uses the method developed by Beuzart-Plessis for the Galois model case in \cite{B17} which does not require the Gelfand pair condition. Then for the space $\CC_{ind}(G(F))$, we first prove the multiplicity one result for all the reduced models by applying the multiplicity formulas of the reduced models (Theorem \ref{main theorem for reduced models}). Then in Appendix \ref{sec:appendix}, by applying the orbit method, we can show that the Gelfand pair condition holds for all tempered representations that are not discrete series (Proposition \ref{Gelfand pair}). This allows us to prove the spectral expansion for the subspace $\CC_{ind}(G(F))$ by applying the same argument as in the GGP case. For details, see Section \ref{sec:proof-spectral}.

Now let us discuss the proof of Theorem \ref{main U} (i.e. the unitary group case). The idea is still to prove a multiplicity formula for all the tempered representations of $G_{1,\varepsilon}(F)$ and then prove the theorem by applying the multiplicity formula together with the behavior of the distribution characters on the local $L$-packet. However, the proof of the multiplicity formula is quite different from all the previous cases. To be specific, in all the previous cases, the proof of the multiplicity formula is based on the proof of a local trace formula for the model. However, for the model $(G_{1,\varepsilon},H_{1,\varepsilon})$, it is not clear to us how to prove the local trace formula. There are two reasons: one is that unlike the previous cases, we have more than one open Borel orbit (in fact, we have two of them) for the model $(G_{1,\varepsilon}, H_{1,\varepsilon})$. The second reason is that when we study the slice representation (i.e. the conjugation action of $H_{1,\varepsilon}(F)$ on the normal space of the spherical variety $G_{1,\varepsilon}(F)/H_{1,\varepsilon}(F)$), the regular orbits do not correspond to the orbits under the $G_{1,\varepsilon}(F)$-conjugation. Some $G_{1,\varepsilon}(F)$-conjugation orbits in the normal space will break into two $H_{1,\varepsilon}(F)$-conjugation orbits (both reasons are related to the fact that there are two elements in the quotient $F^{\times}/\Im(N_{E/F})$). As a result, we have to prove the multiplicity formula by a different method. To be specific, we first prove a relation between the model $(G_{1,\varepsilon},H_{1,\varepsilon})$ and the model $(G_{\varepsilon},H_{\varepsilon})$ (see Proposition \ref{GU v.s. U}). Then we prove the multiplicity formula for the model $(G_{1,\varepsilon},H_{1,\varepsilon})$ by applying the multiplicity formula for the model $(G_{\varepsilon},H_{\varepsilon})$ together with Proposition \ref{GU v.s. U}. For details, see Section \ref{sec:unitary}.

The last thing we want to emphasize about the unitary group case is that in the multiplicity formula for the unitary group case, the regular germ at the identity element has coefficient $2$ and this is why we have the summation of the multiplicities over the $L$-packet is equal to $2$. This is different from all the previous cases. In the GGP case, GR case, and the  unitary similitude group case, the coefficient of the regular germ at the identity element is $1$ and the summation of the multiplicities over the $L$-packet is equal to $1$. While in the Galois model and the generalized Shalika model cases, the multiplicity formulas do not contain the regular germ at the identity element and the multiplicities are constant over the $L$-packet. We believe this new phenomenon should be related to either the fact that there are two open Borel orbits, or to the fact that in the slice representation, some $G_{1,\varepsilon}(F)$-conjugation orbits break into two $H_{1,\varepsilon}(F)$-orbits.

\subsection{Organizations of the paper}
In Section \ref{sec:preliminary}, we introduce basic notation and conventions of this paper. We will also discuss the definitions and some basic facts of the Harish-Chandra-Schwartz space and the strongly cuspidal functions. Then in Section \ref{sec:U-GU}, we discuss some local representation theory of the unitary group and the  unitary similitude group.

In Section \ref{sec:model}, we study the analytic and geometric properties of the model the model $(G_{\varepsilon}, H_{\varepsilon})$. In particular, we show that it is a wavefront spherical variety and has polynomial growth as a homogeneous space.
This gives us the weak Cartan decomposition. Then we discuss some estimates for various integrals which will be used in later sections. The proofs of all the results in this section are very similar to the GGP case (\cite{B15}) and the GR case (\cite{Wan16}), we will skip them here.

In Section \ref{sec:trace-formula}, we will state the trace formulas and the multiplicity formulas for the model $(G_{\varepsilon}, H_{\varepsilon})$ and for its reduced models. In Section \ref{sec:proof-geometric}, we give a sketch of the proof of the geometric side of the trace formula. Since the idea of the proof is similar to the GGP case and all the computations are very similar to the GR case, we will skip the details of the proof. Then by induction, we assume that the multiplicity formulas hold for all reduced models. Finally in Section \ref{sec:consequence}, we discuss some applications of the multiplicity formulas for the reduced models. We will postpone the proof of a technical proposition (i.e. Proposition \ref{Gelfand pair}) to Appendix \ref{sec:appendix}.

In Section \ref{sec:proof-main}, we prove our main theorems by assuming the trace formula holds. Then in Section \ref{sec:proof-spectral}, we will prove the trace formula. Finally, in Appendix \ref{sec:appendix}, we prove the technical proposition in Section \ref{sec:trace-formula} (i.e. Proposition~\ref{Gelfand pair}) by applying the orbit method.

\subsection{Acknowledgement}
We would like to thank Dihua Jiang for suggesting us thinking about this problem. Collaboration on this work started at a workshop in the American Institute of Mathematics, and the work was completed at a workshop in Zhejiang University. We thank these institutions for their hospitality and the organizers for the invitations.

\section{Preliminaries} \label{sec:preliminary}
\subsection{Notation and conventions}
Let $F$ be a $p$-adic field, and let $|\cdot|=|\cdot|_F$ be the absolute value on $F$. For every connected reductive algebraic group $G$ defined over $F$, let $A_G$ be the maximal split center of $G$ and let $Z_G$ be the center of $G$.
We denote by $X(G)$ the group of $F$-rational characters of $G$. Define $\Fa_G=$Hom$(X(G),\BR)$, and let $\Fa_{G}^{\ast}=X(G)\otimes_{\BZ} \BR$ be the dual of $\Fa_G$. We define a homomorphism $H_G:G(F)\rightarrow \Fa_G$ by $H_G(g)(\chi)=\log(|\chi(g)|_F)$ for every $g\in G(F)$ and $\chi\in X(G)$. Let $\Fa_{G,F}$ (resp. $\tilde{\Fa}_{G,F}$) be the image of $G(F)$ (resp. $A_G(F)$) under $H_G$. Then $\Fa_{G,F}$ and $\tilde{\Fa}_{G,F}$ are lattices in $\Fa_G$. Let $\Fa_{G,F}^{\vee}=\Hom(\Fa_{G,F},2\pi \BZ)$ and let $\tilde{\Fa}_{G,F}^{\vee}= \Hom(\tilde{\Fa}_{G,F}, 2\pi \BZ)$. Set $\Fa_{G,F}^{\ast}=\Fa_{G}^{\ast}/\Fa_{G,F}^{\vee}$. We can identify $i\Fa_{G,F}^{\ast}$ with the group of unitary unramified characters of $G(F)$ by letting $\lambda(g)=e^{\langle \lambda,H_G(g)\rangle}$ for $\lambda\in i\Fa_{G,F}^{\ast}$ and $g\in G(F)$. For a Levi subgroup $M$ of $G$, let $\Fa_{M,0}^{\ast}$ be the subset of elements in $\Fa_{M,F}^{\ast}$ whose restriction to $\tilde{\Fa}_{G,F}$ is zero. Then we can identify $i\Fa_{M,0}^{\ast}$ with the group of unitary unramified characters of $M(F)$ which is trivial on $Z_G(F)$.

Let $\Fg$ be the Lie algebra of $G$. For a Levi subgroup $M$ of $G$, let $\CP(M)$ be the set of parabolic subgroups of $G$ whose Levi part is $M$, $\CL(M)$ be the set of Levi subgroups of $G$ containing $M$, and let $\CF(M)$ be the set of parabolic subgroups of $G$ containing $M$. We have a natural decomposition $\Fa_M=\Fa_{M}^{G}\oplus \Fa_G$. Denote by $proj_{M}^{G}$ and $proj_G$ the projections of $\Fa_M$ to each factors. For each $P\in \CP(M)$, we can associate a positive chamber $\Fa_{P}^{+}\subset \Fa_M$, and we can also define a function $H_P:G(F)\rightarrow \Fa_M$ by $H_P(g)=H_M(m_g)$ where $g=m_g u_g k_g$ is the Iwasawa decomposition of $g$.

Let $\Vert \cdot\Vert$ be the height function on $G(F)$, taking values in $\BR_{\geq 1}$. Then we define a log-norm $\sigma$ on $G(F)$ by $\sigma(g)=\sup\{1,\log(\Vert g\Vert)\}$. We also define $\sigma_0(g)=\inf_{z\in Z_G(F)} \{\sigma(zg)\}$. Similarly, we can define the log-norm function on $\Fg(F)$ as follows: fixing a basis $\{X_i\}$ of $\Fg(F)$ over $F$, for $X\in \Fg(F)$, let $\sigma(X)=\sup\{1,\sup\{ \log(|a_i|) \}\}$, where $a_i$ is the $X_i$-coordinate of $X$.

Let $M_{min}$ be a minimal parabolic subgroup of $G$. For each $P_{min}\in \CP(M_{min})$, let $\Psi(A_{min},P_{min})$ be the set of positive roots associated to $P_{min}$, and let $\Delta(A_{min},P_{min})\subset \Psi(A_{min},P_{min})$ be the subset of simple roots.

For $x\in G$ (resp. $X\in \Fg$), let $Z_G(x)$ (resp. $Z_G(X)$) be the centralizer of $x$ (resp. $X$) in $G$, and let $G_x$ (resp. $G_X$) be the neutral component of $Z_G(x)$ (resp. $Z_G(X)$). Accordingly, let $\Fg_x$ (resp. $\Fg_X$) be the Lie algebra of $G_x$ (resp. $G_X$).  Denote by $G_{ss}(F)$ the set of semisimple elements in $G(F)$, and by $G_{reg}(F)$ the set of regular semisimple elements in $G(F)$. The Lie algebra versions
are denoted by $\Fg_{ss}(F)$ and $\Fg_{reg}(F)$, respectively. For $x\in G_{ss}(F)$ (resp. $X\in \Fg_{ss}(F)$), let $D^G(x)$ (resp. $D^G(X)$) be the Weyl determinant.

For two complex valued functions $f$ and $g$ on a set $X$ with $g$ taking values in the positive real numbers, we write
$
f(x)\ll g(x),
$
and say that $f$ is {\sl essentially bounded} by $g$, if there exists a constant $c>0$ such that for all $x\in X$, we have
$
| f(x)| \leq cg(x).
$
We say $f$ and $g$ are {\sl equivalent}, which is denoted by
$f(x)\sim g(x)$,
if $f$ is essentially bounded by $g$ and $g$ is essentially bounded by $f$.

\subsection{Measures}
Through this paper, we fix a non-trivial additive character $\psi: F\rightarrow \BC^{\times}$. If $G$ is a connected reductive group, we may fix a non-degenerate symmetric bilinear form $\langle\cdot,\cdot \rangle$ on $\Fg(F)$ that is invariant under $G(F)$-conjugation. For any smooth compactly supported complex valued function $f\in C_{c}^{\infty}(\Fg(F))$, we can define its Fourier transform $f\rightarrow \hat{f}$ to be
\begin{equation}\label{FT}
\hat{f}(X)=\int_{\Fg(F)} f(Y) \psi(\langle X,Y \rangle) \ud Y
\end{equation}
where ${\rm d} Y$ is the self-dual Haar measure on $\Fg(F)$ such that $\hat{\hat{f}}(X)=f(-X)$. Then we get a Haar measure on $G(F)$ such that the Jacobian of the exponential map equals 1. If $H$ is a subgroup of $G$ such that the restriction of the bilinear form to $\Fh(F)$ is also non-degenerate, then we can define the measures on $\Fh(F)$ and $H(F)$ by the same method.

\subsection{Induced representation}
Given a parabolic subgroup $P=MU$ of $G$ and a smooth admissible representation $(\tau,V_{\tau})$ of $M(F)$, let $(I_{P}^{G}(\tau),I_{P}^{G}(V_{\tau}))$ be the normalized parabolic induced representation: $I_{P}^{G}(V_{\tau})$ is the space of smooth functions $e\colon  G(F)\rightarrow V_{\tau}$ such that
$$e(mug)=\delta_P(m)^{1/2} \tau(m)e(g), \;m\in M(F),\; u\in U(F),\; g\in G(F).$$
And the $G(F)$-action is just the right translation.

For $\lambda\in \Fa_{M}^{\ast}\otimes_{\BR}\BC$, let $\tau_{\lambda}$ be the unramified twist of $\tau$, i.e. $\tau_{\lambda}(m)=\exp(\lambda(H_M(m))) \tau(m)$ and let $I_{P}^{G}(\tau_{\lambda})$ be the induced representation. By the Iwasawa decomposition, every function $e\in I_{P}^{G}(\tau_{\lambda})$ is determined by its restriction on $K$, and that space is invariant under the unramified twist. i.e. for any $\lambda$, we can realize the representation $I_{P}^{G}(\tau_{\lambda})$ on the space $I_{K\cap P}^{K}(\tau_K)$ which consists of functions $e_K:K\rightarrow V_{\tau}$ such that
$$e(mug)=\delta_P(m)^{1/2} \tau(m)e(g), \;m\in M(F)\cap K,\; u\in U(F)\cap K,\; g\in K.$$
Here $\tau_K$ is the restriction of $\tau$ to the group $K\cap M(F)$.

\subsection{Harish-Chandra-Schwartz space}
We use $\Xi^G$ to denote the Harish-Chandra function of $G$. We refer the readers to Proposition 1.5.1 of \cite{B15} for the basic properties of the function $\Xi^G$.

For $f\in C^{\infty}(G(F))$ and $d\in \BR$, let
$$p_d(f)=\sup_{g\in G(F)} \{ |f(g)|\Xi^G(g)^{-1}\sigma(g)^d\}.$$
We define the Harish-Chandra-Schwartz space to be
$$\CC(G(F))=\{f\in C^{\infty}(G(F))\colon p_d(f)<\infty,\forall d>0\}.$$
We also need the weak Harish-Chandra-Schwartz space $\CC^w(G(F))$. For $d>0$, let
$\CC^{w}_{d}(G(F))=\{f\in C^{\infty}(G(F))\colon p_{-d}(f)<\infty\}.$
Then we define
$$\CC^w(G(F))=\cup_{d>0} \CC^{w}_{d}(G(F)).$$

Given a unitary character $\chi$ of $Z_G(F)$,
we define the Harish-Chandra-Schwartz space $\CC(\zg,\chi)$
(resp.\ the weak Harish-Chandra-Schwartz space $\CC^w(\zg,\chi)$) to be the Mellin transform of the space $\CC(G(F))$ (resp. $\CC^w(G(F))$) with respect to $\chi$.

\subsection{The Harish-Chandra-Plancherel formula}
Fix a unitary character $\chi$ of $Z_G(F)$. For every $M\in \CL(M_{min})$, fix an element $P\in \CP(M)$. Let $\Pi_2(M,\chi)$ be the set of discrete series of $M(F)$ whose central character agrees with $\chi$ on $Z_G(F)$. Then $i\Fa_{M,0}^{\ast}$ acts on $\Pi_2(M,\chi)$ by the unramified twist. Let $\{\Pi_2(M,\chi)\}$ be the set of orbits under this action. For every orbit $\CO$, and for a fixed $\tau\in \CO$, let $i\Fa_{\CO}^{\vee}$ be the set of $\lambda\in i\Fa_{M,0}^{\ast}$ such that the representation $\tau$ and $\tau_{\lambda}$ are equivalent, which is a finite set. For $f\in \CC(\zg,\chi^{-1})$, the Harish-Chandra-Plancherel formula (\cite{W03}) is
\begin{eqnarray*}
f(g)&=&\sum_{M\in \CL(M_{min})} |W^M||W^G|^{-1} \sum_{\CO\in \{\Pi_2(M,\chi)\}} |i\Fa_{\CO}^{\vee}|^{-1}\\
&&\int_{i\Fa_{M,0}^{\ast}} \mu(\tau_{\lambda}) \tr(I_{P}^{G}(\tau_{\lambda})(g^{-1}) I_{P}^{G}(\tau_{\lambda})(f))\ud\lambda.
\end{eqnarray*}
Here $\mu(\tau_{\lambda})$ is the Plancherel measure and $W^G$ (resp. $W^M$) is the Weyl group of $G$ (resp. $M$).

To simplify our notation, let $\Pi_{temp}(G,\chi)$ be the union of $I_{P}^{G}(\tau)$ for $P=MN$, $M\in \CL(M_{min})$, $\tau\in \CO$ and $\CO\in \{\Pi_2(M,\chi)\}$. We define a Borel measure ${\rm d}\pi$ on $\Pi_{temp}(G,\chi)$ such that
\begin{align*}
 & \int_{\Pi_{temp}(G,\chi)}\varphi(\pi)\ud\pi\\
=&\sum_{M\in \CL(M_{min})} |W^M|\cdot |W^G|^{-1} \sum_{\CO\in \{\Pi_2(M,\chi)\}} |i\Fa_{\CO}^{\vee}|^{-1} \int_{i\Fa_{M,0}^{\ast}} \varphi(I_{P}^{G}(\tau_{\lambda})) \ud\lambda
\end{align*}
for every compactly supported function $\varphi$ on $\Pi_{temp}(G,\chi)$. Here by saying a function $\varphi$ is compactly supported on $\Pi_{temp}(G,\chi)$ we mean that it is supported on finitely many orbits $\CO$. Then the Harish-Chandra-Plancherel formula above becomes
$$f(g)=\int_{\Pi_{temp}(G,\chi)} \tr(\pi(g^{-1}) \pi(f)) \mu(\pi)\ud\pi.$$

We also need the metrical Paley-Wiener Theorem. Let $C^{\infty}(\Pi_{temp}(G,\chi))$ be the space of functions $\pi\in \Pi_{temp}(G,\chi)\rightarrow T_{\pi}\in \End(\pi)^{\infty}$ such that it is smooth on every orbits $\CO$ as functions from $\CO$ to $\End(\pi)^{\infty}\simeq \End(\pi_K)^{\infty}$. We define $\CC(\Pi_{temp}(G,\chi))$ to be a subspace of $C^{\infty}(\Pi_{temp}(G,\chi))$ consisting of those $T:\pi\rightarrow T_{\pi}$ such that $T$ is nonzero on finitely many orbits $\CO$. Then the metrical Paley-Wiener Theorem (\cite{W03}) states that we have an isomorphism between $\CC(\zg,\chi^{-1})$ and $\CC(\Pi_{temp}(G,\chi))$ given by
$$f\in \CC(\zg,\chi^{-1})\rightarrow T_f:=(\pi\in \Pi_{temp}(G,\chi)\rightarrow \pi(f)\in \End(\pi)^{\infty}),$$
$$T\in \CC(\Pi_{temp}(G,\chi))\rightarrow f_T(g)=\int_{\Pi_{temp}(G,\chi)} \tr(\pi(g^{-1}) T_{\pi}) \mu(\pi) \ud\pi.$$

Finally, we introduce two subspaces of $\CC(\zg,\chi^{-1})$. Let $\Pi_{2}(G,\chi)$  be a subset of $\Pi_{temp}(G, \chi)$ consisting of all the discrete series, and let $\Pi_{temp, ind}(G,\chi)=\Pi_{temp}(G,\chi)\smallsetminus\Pi_{2}(G,\chi)$. Let $\CC(\Pi_{2}(G,\chi))$ (resp. $\CC(\Pi_{temp,ind}(G,\chi))$) be a subspace of $\CC(\Pi_{temp}(G,\chi))$ consisting of those $T:\pi\rightarrow T_{\pi}$ such that $T$ is supported on the set $\Pi_{2}(G,\chi)$ (resp. $\Pi_{temp, ind}(G,\chi)$). Then any element $T\in \CC(\Pi_{temp}(G,\chi))$ can be uniquely written as $T=T_1+T_2$ with $T_1\in \CC(\Pi_{2}(G,\chi))$ and $T_2\in \CC(\Pi_{temp,ind}(G,\chi))$. In other words, we have $\CC(\Pi_{temp}(G,\chi)) =\CC(\Pi_{2}(G,\chi))\oplus \CC(\Pi_{temp,ind}(G,\chi))$.

Under the metrical Paley-Wiener Theorem, the subspaces $\CC(\Pi_{2}(G,\chi))$ and $\CC(\Pi_{temp,ind}(G,\chi))$ of $\CC(\Pi_{temp}(G,\chi))$ allow us to define the corresponding subspaces of $\CC(\zg,\chi^{-1})$. We define
\begin{align*}
{}^{\circ}\CC(\zg,\chi^{-1})=&\{f\in \CC(\zg,\chi^{-1})\colon T_f\in \CC(\Pi_{2}(G,\chi))\},\\
\CC_{ind}(\zg,\chi^{-1})=&\{f\in \CC(\zg,\chi^{-1})\colon T_f\in \CC(\Pi_{temp,ind}(G,\chi))\}.
\end{align*}
Then we have
$$\CC(\zg,\chi^{-1})={}^{\circ}\CC(\zg,\chi^{-1})\oplus \CC_{ind}(\zg,\chi^{-1}).$$
It is easy to see that the space ${}^{\circ}\CC(\zg,\chi^{-1})$ is spanned by the matrix coefficients of all the discrete series of $G(F)$ with central character $\chi^{-1}$.

\subsection{Quasi-characters}
If $\theta$ is a smooth function defined on $G_{reg}(F)$, invariant under $G(F)$-conjugation. We say it is a quasi-character on $G(F)$ if for every $x\in G_{ss}(F)$, there is a good neighborhood $\omega_x$ of $0$ in $\Fg_x(F)$, and for every $\CO\in Nil(\Fg_x)$, there exists $c_{\theta,\CO}(x)\in \BC$ such that
\begin{equation}\label{germ 2}
\theta(x\exp(X))=\sum_{\CO\in Nil(\Fg_x)} c_{\theta,\CO}(x) \hat{j}(\CO,X)
\end{equation}
for every $X\in \omega_{x,reg}$. Here $\hat{j}(\CO,X)$ is the function on $\Fg_{reg}(F)$ representing the Fourier transform of the nilpotent orbital integral, $Nil(\Fg_x)$ is the set of nilpotent orbits of $\Fg_x(F)$, and we refer the readers to Section 3 of \cite{W10} for
the definition of good neighborhood. The coefficients $c_{\theta,\CO}(x)$ are called the germs of $\theta$ at $x$. We define
$$c_{\theta}(x)=\frac{1}{|Nil_{reg}(\Fg_x)|} \sum_{\CO\in Nil_{reg}(\Fg_x)} c_{\theta,\CO}$$
to be the average of the germs associated to the regular nilpotent orbits of $\Fg_x$. In particular, $c_{\theta}(x)=0$ if $G_x$ is not quasi-split. For any admissible representation $\pi$ of $G(F)$, the distribution character $\theta_{\pi}$ is a quasi-character.

Similarly, if $\theta$ is a smooth function on $\Fg_{reg}(F)$ and invariant under $G(F)$-conjugation,
we say it is a quasi-character on $\Fg(F)$ if for every $X\in \Fg_{ss}(F)$, there exists an open $G_X$-domain $\omega_X$ in $\Fg_X(F)$, containing $0$, and for every $\CO \in Nil(\Fg_X)$, there exists $c_{\theta,\CO}(X)\in \BC$ such that
\begin{equation}\label{germ 3}
\theta(X+Y)=\sum_{\CO\in Nil(\Fg_X)} c_{\theta,\CO}(X) \hat{j}(\CO,Y)
\end{equation}
for every $Y\in \omega_{X,reg}$. As in the group case, we use $c_{\theta}(X)$ to denote the average of the germs associated to the regular nilpotent orbits.

\subsection{Strongly cuspidal functions}
We say a function $f\in \CC(Z_G(F)\backslash G(F),\chi)$ is strongly cuspidal if for every proper parabolic subgroup $P=MU$ of $G$, and for every $x\in M(F)$, we have
\begin{equation}\label{cuspidal 1}
\int_{U(F)} f(xu)\ud u=0.
\end{equation}
We will denote by $\CC_{scusp}(Z_G(F)\backslash G(F),\chi)$ (resp. $C_{c,scusp}^{\infty}(Z_G(F)\backslash G(F),\chi)$) the subspace of strongly cuspidal functions in $\CC(Z_G(F)\backslash G(F),\chi)$ (resp. $C_{c}^{\infty}(Z_G(F)\backslash G(F),\chi)$). It is easy to see that ${}^{\circ}\CC(Z_G(F)\backslash G(F),\chi)\subset \CC_{scusp}(Z_G(F)\backslash G(F),\chi)$. Hence we have
$$\CC_{scusp}(Z_G(F)\backslash G(F),\chi)={}^{\circ}\CC(Z_G(F)\backslash G(F),\chi)\oplus \CC_{ind,scusp}(Z_G(F)\backslash G(F),\chi)$$
where $\CC_{ind,scusp}(Z_G(F)\backslash G(F),\chi)$ is the subspace of strongly cuspidal functions in $\CC_{scusp}(Z_G(F)\backslash G(F),\chi)$.

Similarly, we say a function $f\in C_{c}^{\infty}(\Fg(F))$ is strongly cuspidal if for every proper parabolic subgroup $P=MU$, and for every $X\in \Fm(F)$, we have
$$\int_{\Fu(F)} f(X+Y)\ud Y=0.$$

We then define various objects associated to strongly cuspidal functions. Geometrically, for $f\in \CC_{scusp}(\zg,\chi)$ (resp. $f\in C_{c,scusp}^{\infty}(\Fg(F))$), one can define a quasi-character $\theta_f$ of $G(F)$ (resp. $\Fg(F)$) via the weighted orbital integral. We refer the readers to Section 5.2 of \cite{B15} for details of the definition. Spectrally, let $\CX(G,\chi^{-1})$ (resp. $\CX_{ell}(G,\chi^{-1})$) be a set of virtual tempered representations (resp. elliptic representations) of $G(F)$ with central character $\chi^{-1}$ defined in Section 2.7 of \cite{B15}. As in Section 5.4 of loc. cit., for $\pi\in \CX(G,\chi^{-1})$, we can define a map
$$f\in \CC_{scusp}(\zg,\chi)\mapsto \theta_f(\pi)\in \BC$$
via the weighted character (this map is denoted by $f\mapsto \hat{\theta}_f(\pi)$ in loc. cit.). We refer the readers to Section 5 of \cite{B15} for basic properties of strongly cuspidal functions.

\subsection{Some local representation theory of the groups $\RU_{2n}$ and $\GU_{2n}$} \label{sec:U-GU}
In this subsection, we recall some results of the local representation theory of the unitary group and the unitary similitude  group. For $i=1,2$, fix $\varepsilon_i\in F^{\times}$ with $\eta_{E/F}(\varepsilon_i)=(-1)^{i-1}$ as before. We start with the unitary group case. The following theorem follows from the endoscopic classification of unitary group in \cite{M15} and \cite{KMSW}.
Denote $\Pi_{irr,temp}(\RU_{2n})=\Pi_{irr,temp}(\RU(J_{2n,\varepsilon_1}))\cup \Pi_{irr,temp}(\RU(J_{2n,\varepsilon_2}))$ to be set of all the irreducible tempered representations of $\RU(J_{2n,\varepsilon_1})(F)$ and $\RU(J_{2n,\varepsilon_2})(F)$.

\begin{thm}[\cite{M15}, \cite{KMSW}]\label{endoscopy U}
$\Pi_{irr,temp}(\RU_{2n})$ is a disjoint union of finite sets (i.e. the local tempered Vogan $L$-packets)
$$\Pi_{irr,temp}(\RU_{2n})=\cup_{\phi} \Pi_{\phi}$$
where $\phi$ runs over all the tempered $L$-parameters of $\RU_{2n}(F)$ and $\Pi_{\phi}=\Pi_{\phi}(\RU(J_{2n,\varepsilon_1}))\cup \Pi_{\phi}(\RU(J_{2n,\varepsilon_2}))$ consisting of a finite number of tempered representations such that the following conditions hold.
\begin{enumerate}
\item \label{item:endoscopy-1} For all $\phi$, the distribution character
$$\theta_{\Pi_{\phi}(\RU(J_{2n,\varepsilon_i}))}:= \sum_{\pi_{\varepsilon_i}\in \Pi_{\phi}(\RU(J_{2n,\varepsilon_i}))} \theta_{\pi_{\varepsilon_i}}$$
is stable for $i=1,2$.
\item \label{item:generic} Let $\psi$ be any generic character of the maximal unipotent subgroup of $\RU(J_{2n,\varepsilon_1})$ (up to conjugation, there are two such characters). For all $\phi$, the $L$-packet $\Pi_{\phi}(\RU(J_{2n,\varepsilon_1}))$ contains a unique generic representation with respect to $\psi$ (note that $\RU(J_{2n,\varepsilon_2})(F)$ is not quasi-split).
\item \label{item:endoscopy-3} For a pair $g_i\in \RU(J_{2n,\varepsilon_i})(F)$, we write $g_1\leftrightarrow g_2$ if they have the same characteristic polynomial (i.e. they are in the same stable conjugacy class). Then for any pair $g_i\in \RU(J_{2n,\varepsilon_i})(F)_{reg}$ with $g_1\leftrightarrow g_2$, we have
    $$\theta_{\Pi_{\phi}(\RU(J_{2n,\varepsilon_1}))}(g_1)=-\theta_{\Pi_{\phi}(\RU(J_{2n,\varepsilon_2}))}(g_2).$$
\end{enumerate}
\end{thm}
Remark that Item \eqref{item:generic} of the theorem was proved in \cite{K02} by assuming the endoscopic identity holds. The endoscopic identity for unitary group has been proved in \cite{M15} and \cite{KMSW}.

Now we consider the unitary similitude group case. Recall that for $\varepsilon\in F^{\times}$, $\GU_{2n,\varepsilon}=\GU (J_{2n,\varepsilon})$ and $\RU_{2n,\varepsilon}=\RU (J_{2n,\varepsilon})$. We have
$$\GU_{2n,\varepsilon}(F)/Z_{\GU_{2n,\varepsilon}}(F)\RU_{2n,\varepsilon}(F)\cong \lam(\GU_{2n,\varepsilon}(F))/\lam(Z_{\GU_{2n,\varepsilon}}(F)) =F^\times/\Im(N_{E/F}) \cong \BZ/2\BZ,$$
$$X:=\Hom(\GU_{2n,\varepsilon}(F)/Z_{\GU_{2n,\varepsilon}}(F)\RU_{2n,\varepsilon}(F),\BC^\times)=\{1, \lambda_{E/F}\}$$
where $\lambda_{E/F}=\eta_{E/F}\circ \lambda$ is a character of $\GU_{2n,\varepsilon}(F)$. The proofs of following three lemmas can be found in \cite{Xu16} for instance.

\begin{lem}[Corollary 6.7 \cite{Xu16}]\label{GU to U 1}
Let $\pi_\varepsilon$ be an irreducible smooth representation of $\GU_{2n,\varepsilon}(F)$ and let $\pi_\varepsilon\vert_{\RU_{2n,\varepsilon} (F)}$ be the restriction of $\pi_\varepsilon$ to $\RU_{2n,\varepsilon}(F)$. Then $\pi_\varepsilon\vert_{\RU_{2n,\varepsilon}(F)}$ is multiplicity-free. Moreover,  $\pi_\varepsilon\vert_{\RU_{2n,\varepsilon}(F)}$ is irreducible if and only if  $\pi_\varepsilon\ncong \pi_\varepsilon\otimes \lambda_{E/F}$.
If $\pi_\varepsilon\cong \pi_\varepsilon\otimes \lambda_{E/F}$, then $\pi_\varepsilon\vert_{\RU_{2n,\varepsilon}(F)}=\pi\oplus\pi\circ\Ad(g)$ for some irreducible representation $\pi$ of $\RU_{2n,\varepsilon}(F)$, where $g\in \GU_{2n,\varepsilon}(F)$ with $\lambda(g)\notin \Im(N_{E/F})$.
\end{lem}

\begin{lem}[Corrollay 6.4 \cite{Xu16}]\label{GU to U 2}
If $\pi$ is an irreducible smooth representation of $\RU_{2n,\varepsilon}(F)$, then there exists an irreducible smooth representation $\pi_{\varepsilon}$ of $\GU_{2n,\varepsilon}(F)$, which is unique up to twisting by the characters $\chi\circ \lambda$ where $\chi$ is any character of $F^{\times}$, such that $\pi$ is a direct summand of $\pi_\varepsilon\vert_{\RU_{2n,\varepsilon}(F)}$.
\end{lem}

\begin{lem}[Lemma 6.9 \cite{Xu16}]\label{GU to U 3}
Suppose that $\pi_{\varepsilon}$ is an irreducible smooth unitary representation of $\GU_{2n,\varepsilon}(F)$. Then $\pi_{\varepsilon}$ is a discrete series if and only if its restriction to $\RU_{2n,\varepsilon}(F)$ is a discrete series (not necessarily irreducible).
The same is true for tempered representations.
\end{lem}

In order to prove our main theorems, we need to assume that following conjecture holds. This conjecture is the endoscopy classification of the  unitary similitude groups. By the endoscopic classification of the unitary groups (\cite{M15}, \cite{KMSW}), this conjecture is expected from Xu's work \cite{Xu16} on the reduction from the similitude classical groups to classical groups.
Denote $\Pi_{irr,temp}(\GU_{2n})=\Pi_{irr,temp}(\GU(J_{2n,\varepsilon_1}))\cup \Pi_{irr,temp}(\GU(J_{2n,\varepsilon_2}))$ to be the set of all the irreducible admissible tempered representations of $\GU(J_{2n,\varepsilon_1})(F)$ and $\GU(J_{2n,\varepsilon_2})(F)$.

\begin{conj}\label{endoscopy}
$\Pi_{irr,temp}(\GU_{2n})$ is a disjoint union of finite sets (i.e. the local tempered Vogan $L$-packets)
$$\Pi_{irr,temp}(\GU_{2n})=\cup_{\phi} \Pi_{\phi}$$
where $\phi$ runs over all the tempered $L$-parameters of $\GU_{2n}(F)$ and $\Pi_{\phi}=\Pi_{\phi}(\GU(J_{2n,\varepsilon_1}))\cup \Pi_{\phi}(\GU(J_{2n,\varepsilon_2}))$ consisting of a finite number of tempered representations such that the following conditions hold.
\begin{enumerate}
\item  For all $\phi$, the distribution character
$$\theta_{\Pi_{\phi}(\GU(J_{2n,\varepsilon_i}))}:= \sum_{\pi_{\varepsilon_i}\in \Pi_{\phi}(\GU(J_{2n,\varepsilon_i}))} \theta_{\pi_{\varepsilon_i}}$$ is stable for $i=1,2$.
\item \label{item:conjecture-2} For all $\phi$, the $L$-packet $\Pi_{\phi}(\GU(J_{2n,\varepsilon_1}))$ contains a unique generic representation (note that $\GU(J_{2n,\varepsilon_2})(F)$ is not quasi-split).
\item \label{item:conjecture-3} For $g_i\in \GU(J_{2n,\varepsilon_i})(F)$, we write $g_1\leftrightarrow g_2$ if they have the same characteristic polynomial (i.e. they are in the same stable conjugacy class). Then
    $$\theta_{\Pi_{\phi}(\GU(J_{2n,\varepsilon_1}))}(g_1)=-\theta_{\Pi_{\phi}(\GU(J_{2n,\varepsilon_2}))}(g_2)$$
    for all pairs $g_i\in \GU(J_{2n,\varepsilon_i})(F)_{reg}$ with $g_1\leftrightarrow g_2$.
\end{enumerate}
\end{conj}

\section{The model $(G_{\varepsilon},H_{\varepsilon})$} \label{sec:model}
\subsection{The spherical pair $(G_{\varepsilon},H_{\varepsilon})$}
Let $(G,H_{0}U)$ be the pair $(G_{\varepsilon},H_{0,\varepsilon}U_{\varepsilon})$ defined in Section \ref{sec:introduction}  for some $\varepsilon\in F^{\times}$, and let $G_0=M_{\varepsilon}$.
Then the spherical pair $(G,H)$ can be viewed as the parabolic induction of the spherical pair $(G_0,H_0)$. We will use $\omega\otimes\xi$ to denote the character $\omega_{\varepsilon}\otimes \xi_{\varepsilon}$.
For simplicity, we omit the subscript $\varepsilon$ here.
We say a parabolic subgroup $\bar{Q}$ of $G$ is good if $H\bar{Q}$ is a Zariski open subset of $G$. This is equivalent to say that $H(F)\bar{Q}(F)$ is open in $G(F)$ under the analytic topology. The proof of the next proposition is very similar to the GR model case (Proposition 4.2 of \cite{Wan16}), so we will skip it here. The only thing we want to point out is that the proposition will only hold for the  unitary similitude group case as we will have two open Borel orbits (which correspond to $F^{\times}/\Im(N_{E/F})$) for the unitary group case.

\begin{prop}\label{spherical}
\begin{enumerate}
\item  \label{item:prop-spherical-1}  There exist minimal parabolic subgroups of $G$ that are good and they are all conjugated to each other by some elements in $H(F)$. If $\bar{P}_{min}=M_{min}\bar{U}_{min}$ is a good minimal parabolic subgroup, we have $H\cap \bar{U}_{min}=\{1\}$ and the complement of $H(F)\bar{P}_{min}(F)$ in $G(F)$ has zero measure.
\item  \label{item:prop-spherical-2} A parabolic subgroup $\bar{Q}$ of $G$ is good if and only if it contains a good minimal parabolic subgroup.
\item  \label{item:prop-spherical-3} Let $\bar{P}_{min}=M_{min}\bar{U}_{min}$ be a good minimal parabolic subgroup and let $A_{min}=A_{M_{min}}$ be the split center of $M_{min}$. Set
$$A_{min}^{+}=\{a\in A_{min}(F)\colon |\alpha(a)| \geq 1 \; \text{for}\; \text{all} \; \alpha\in \Psi(A_{min},\bar{P}_{min})\}.$$
Then we have
\begin{enumerate}
\item $\sigma_0(h)+\sigma_0(a)\ll \sigma_0(ha)$ for all $a\in A_{min}^{+}$, $h\in H(F)$.
\item $\sigma(h)\ll \sigma(a^{-1}ha)$ and $\sigma_0(h)\ll \sigma_0(a^{-1}ha)$ for all $a\in A_{min}^{+}$, $h\in H(F)$.
\end{enumerate}
\item \eqref{item:prop-spherical-1}, \eqref{item:prop-spherical-2} and \eqref{item:prop-spherical-3} also hold for the pair $(G_0,H_0)$.
\end{enumerate}
\end{prop}

By the proposition above, $X=H\backslash G$ is a spherical variety of $G$ and $X_0=H_0\backslash G_0$ is a spherical variety of $G_0$. Let $\bar{P}_0=M_0\bar{U}_0$ be a good minimal parabolic subgroup of $G_0$, and let $A_0=A_{M_0}$ be the maximal split center of $M_0$. Set
$$A_{0}^{+}=\{a\in A_0(F)\colon |\alpha(a)| \geq 1, \; \forall \alpha\in \Psi(A_0,\bar{P}_0) \}.$$
By a similar argument as in the Ginzburg-Rallis model case (Proposition 4.4 of \cite{Wan16}), we can prove the weak Cartan decomposition of $X$ and $X_0$.

\begin{prop}\label{cartan 1}
\begin{enumerate}
\item There exists a compact subset $\CK_0\subset G_0(F)$ such that $G_0(F)=H_0(F)A_{0}^{+} \CK_0$.
\item There exists a compact subset $\CK\subset G(F)$ such that $G(F)=H(F)A_{0}^{+} \CK$.
\end{enumerate}
\end{prop}

\subsection{Some estimates}
In this subsection, we are going to state several estimates for various integrals which will be used in later sections. The proofs of these estimates are very similar to the GR model case (Sections 4.3 and 4.4 of \cite{Wan17}). We will skip them here.

\begin{lem}\label{major 2}
\begin{enumerate}
\item There exist $\epsilon>0$ and $d>0$ such that the integrals
$$\int_{Z_{H_0}(F)\backslash H_0(F)} \Xi^{G_0}(h_0)e^{\epsilon \sigma_0(h_0)} \ud h_0
\text{ and }
\int_{Z_H(F)\backslash H(F)} \Xi^G(h) \sigma_0(h)^{-d} \ud h$$
are absolutely convergent.
\item For all $\delta>0$, there exists $\epsilon>0$ such that the integral
$$\int_{\zh} \Xi^{G}(h)e^{\epsilon \sigma_0(h)} (1+ |\iota(h)|)^{-\delta}\ud h$$
is absolutely convergent. Here $\iota\colon H(F)\rightarrow F$ is a homomorphism defined by $\iota(\left(\begin{smallmatrix} I_2&X&Y\\&I_{2}&X^{\ast} \\&&I_2	\end{smallmatrix}\right) \left(\begin{smallmatrix}g&&\\&h&\\&&\lambda(h)g^*	\end{smallmatrix}\right))=\tr_{E/F}(\tr(X))$. In particular, we have $\xi(h)=\psi(\iota(h))$ for all $h\in H(F)$.
\end{enumerate}
\end{lem}

Let $C\subset G(F)$ be a compact subset with non-empty interior. Define the function
$\Xi^{H\backslash G}_{C}(x)=vol_{H\back G}(xC)^{-1/2}$ for $x\in H(F)\back G(F)$.
If $C'$ is another compact subset with non-empty interior, then
$\Xi^{H\backslash G}_{C}(x)\sim \Xi^{H\backslash G}_{C'}(x)$
for all $x\in H(F)\back G(F)$. We will only use the function $\Xi^{H\backslash G}_{C}$ for majorization. From now on, we will fix a particular $C$, and set
$\Xi^{H\backslash G}=\Xi^{H\backslash G}_{C}.$
The next proposition gives some basic properties for the function $\Xi^{H\backslash G}$.

\begin{prop}\label{major 4}
\begin{enumerate}
\item There exists $d>0$ such that the integral
$$\int_{H(F)\back G(F)}\hc(x)^2 \nor(x)^{-d} \ud x$$
is absolutely convergent. Here $\nor(x):=\inf_{h\in H(F)} \sigma(hx)$ for $x\in H(F)\back G(F)$.
\item For all $d>0$, there exists $d'>0$ such that
$$\int_{\zh}\Xi^G(hx)\sigma_0(hx)^{-d'}\ud h\ll \hc(x) \nor(x)^{-d}$$
for all $x\in H(F)\back G(F)$.
\end{enumerate}
\end{prop}

\begin{prop}\label{major 5}
Let $\bar{Q}=M_Q \bar{U}_Q$ be a good parabolic subgroup of $G$. Let $H_{\bar{Q}}=H\cap \bar{Q}$, and let $G_{\bar{Q}}=\bar{Q}/\bar{U}_{Q}$ be the reductive quotient of $\bar{Q}$. Then
\begin{enumerate}
\item $H_{\bar{Q}}\cap \bar{U}_{Q}=\{1\}$, hence we can view $H_{\bar{Q}}$ as a subgroup of $G_{\bar{Q}}$. We also have $\delta_{\bar{Q}}(h_{\bar{Q}})=\delta_{H_{\bar{Q}}} (h_{\bar{Q}})$ for all $h_{\bar{Q}}\in H_{\bar{Q}}(F)$.
\item \label{item:major-5-2} There exists $d>0$ such that the integral
$$\int_{Z_{H}(F)\back H_{\bar{Q}}(F)} \Xi^{G_{\bar{Q}}}(h_{\bar{Q}})\sigma_0(h_{\bar{Q}})^{-d} \delta_{H_{\bar{Q}}}(h_{\bar{Q}})^{1/2} \ud h_{\bar{Q}}$$
is absolutely convergent.
\end{enumerate}
\end{prop}

\section{The trace formula and the multiplicity formula} \label{sec:trace-formula}
\subsection{The distribution $I_{geom}(f)$}
As in the previous section, we will use $(G,H)$ to denote the pair $(G_{\varepsilon},H_{\varepsilon})$ and $\omega\otimes \xi$ to denote the character $\omega_{\varepsilon}\otimes \xi_{\varepsilon}$. Given $f\in \CC_{scusp}(\zg,\eta^{-1})$, we have associated the quasi-character $\theta_f$ on $G(F)$ in Section \ref{sec:preliminary}. Let $\CT_{ell}(H_0)$ be a set of representatives of maximal elliptic tori in $H_0$. In both cases (i.e. the quasi-split case and the non quasi-split case), there is a natural bijection between $\CT_{ell}(H_0)$ and the set of all the quadratic extensions of $F$. Also it is easy to check that for all $T\in \CT_{ell}(H_0)$ and $t\in T(F)_{reg}$, the Lie algebra of $G_t$ has a unique regular nilpotent orbit which will be denoted by $\CO_t$. We then define
$$c_f(t):=c_{\theta_f}(t)=c_{\theta_f,\CO_t}(t).$$

\begin{prop}\label{convergence of geometric expansion}
With the notation above, the integral
$$\int_{T(F)/Z_G(F)} D^H(t)^{1/2} c_f(t) \omega(t)\ud t$$
is absolutely convergent.
\end{prop}

\begin{proof}
The proof is similar to the GR case in Proposition 5.2 of \cite{Wan15}. We will skip it here.
\end{proof}

\begin{defn}\label{definition of geometric side}
We define the geometric side of the trace formula to be
\begin{align*}
I_{geom}(f):=c_{f}(1)+\sum_{T\in \CT_{ell}(H_0)} & |W(H_0,T)|^{-1} \vol(T(F)/Z_G(F))^{-1}\\
 &\times\int_{T(F)/Z_G(F)} D^H(t)^{1/2} c_f(t) \omega(t) \ud t
\end{align*}
where $c_f(1)=c_{\theta_f}(1)$ is the regular germ of $\theta_f$ at $1$. Note that if the group is not quasi-split, $c_f(1)$ is always equal to $0$.
\end{defn}

\subsection{The distribution $I(f)$ and $I_{spec}(f)$}
For $f\in \CC_{scusp} (Z_G(F)\backslash G(F), \eta^{-1})$, define the function $I(f,\cdot)$ on $H(F)\back G(F)$ to be
$$I(f,g)=\int_{\zh} f(g^{-1}hg) \omega\otimes\xi(h) \ud h.$$
By Lemma \ref{major 2}, the above integral is absolutely convergent. Then by the same argument as in the GGP case (Proposition 7.1.1 of \cite{B15}) and the GR case (Appendix B of \cite{Wan17}), we can show that the integral
\begin{equation}\label{I(f)}
I(f):=\int_{H(F)\back G(F)} I(f,g) \ud g
\end{equation}
is absolutely convergent for all $f\in \CC_{scusp}(\zg,\eta^{-1})$, and it defines a continuous linear form
$$\CC_{scusp}(\zg,\eta^{-1})\rightarrow \BC \colon f\rightarrow I(f).$$
$I(f)$ will be the distribution in our trace formula.

\begin{rmk}
As in the GGP case and the GR case, although the integral \eqref{I(f)} defining $I(f)$ is absolutely convergent, the double integral
$$\int_{H(F)\back G(F)}\int_{\zh} f(g^{-1}hg) \omega\otimes\xi(h) \ud h \ud g$$
is not absolutely convergent. As a result, in the proof of the geometric side of the trace formula, we need to introduce truncated functions on $H(F)\back G(F)$.
\end{rmk}

We then define the spectral side of the trace formula to be
\begin{equation}\label{I_{spec}(f)}
I_{spec}(f)=\int_{\CX(G,\eta)} D(\pi)\theta_f(\pi)m(\bar{\pi})\ud\pi.
\end{equation}
We refer the readers to Section 2.7 of \cite{B15} for the definitions of $D(\pi)$ and the measure ${\rm d}\pi$. Now we are ready to state the trace formula.

\begin{thm}\label{trace formula}
For all $f\in \CC_{scusp} (Z_G(F)\backslash G(F),\eta^{-1})$, we have
$$I_{spec}(f)=I(f)=I_{geom}(f).$$
\end{thm}
The spectral expansion will be proved in Section \ref{sec:proof-spectral}, while the geometric expansion will be proved in the next subsection.

To end this subsection, we define the Lie algebra analogy of the distribution $I(f)$ in the trace formula. This will be used in the proof of the geometric expansion.
Denote $\Fg'(F)$ (resp. $\Fh'(F)$) to be the subspace of $\Fg(F)$ (resp. $\Fh(F)$) consisting of elements of trace zero.
Then $\Fg(F)=\Fg'(F)\oplus \Fz_{\Fg}(F)$ and $\Fh(F)=\Fh'(F)\oplus \Fz_{\Fg}(F)$. For $\phi\in C_{c,scusp}^{\infty} (\Fg'(F))$, we define
$$I(\phi,g)=\int_{\Fh'(F)} \phi(g^{-1}Xg) \ud X
\text{ and }
I(\phi)=\int_{H(F)\back G(F)} I(\phi,g)\ud g.$$
As in the group case, the integral defining $I(\phi)$ is absolutely convergent.

\subsection{The proof of the geometric expansion} \label{sec:proof-geometric}
In this subsection, we prove the geometric side of the trace formula. The idea of the proof is the same as the GGP case in \cite{W10} and \cite{B15}, while all the computations are very similar to the GR case in \cite{Wan15}. As a result, we will only give a sketch of the proof without providing details.

First by the standard argument as in the GGP case (Section 11.3 of \cite{B15}), once we have proved the spectral side of the trace formula (this will be done in Section \ref{sec:proof-spectral}), we only need to prove the geometric side for compactly supported functions (i.e. for $f\in C_{c,scusp}^{\infty}(\zg, \eta^{-1})$). Then by the same argument as in the GR case (Proposition 5.6 of \cite{Wan15}), it is enough to consider the case when the characters $\chi_E,\chi_F$ and $\eta$ are trivial.

The next step is to study the distribution $I(\phi)$ for the Lie algebra case (i.e. for $\phi\in C_{c,scusp}^{\infty} (\Fg'(F))$). The goal is to express $I(\phi)$ in terms of $\theta_{\hat{\phi}}=\hat{\theta}_\phi$ where $\hat{\phi}$ is the Fourier transform of $\phi$. In order to do this, we first need to introduce a sequence of truncated functions $\kappa_N\in C_{c}^{\infty}(H(F)\back G(F))$ (where $N\geq 1$) whose definition is similar to the GR case. For $N\geq 1$, we define
$$I_N(\phi)=\int_{H(F)\back G(F)} \kappa_N(g) I(\phi,g)\ud g.$$
We have $I(\phi)=\lim_{N\rightarrow \infty}I_N(\phi).$
Hence, it is enough to consider $I_N(\phi)$.

Then we study the slice representation which is the conjugation action of $H(F)/Z_G(F)$ on the space $\Xi+\Fh^{\perp}(F)$. Here $\Fh^{\perp}$ is the orthogonal complement of $\Fh$ in $\Fg$ and $\Xi=\left(\begin{smallmatrix}0&0&0\\I_2&0&0\\0&-w_2 A_{\varepsilon}&0\end{smallmatrix}\right)$ is an element in $\bar{\Fu}(F)$ associated to the character $\xi$ of $U(F)$. By a very similar computation as in the GR case (Section 8 of \cite{Wan15}), we can show that over a Zariski open subset, this action is free and the orbits are the same as the $G(F)$-conjugacy classes in $\Fg(F)$ (i.e. two elements in $\Xi+\Fh^{\perp}(F)$ are conjugated to each other under $H(F)$ if and only if they are conjugated to each other under $G(F)$). Note that this will fail for the unitary group case (i.e. the model $(G_{1,\varepsilon},H_{1,\varepsilon})$), and it is one of the reasons why we cannot prove the trace formula for the unitary group case.

After studying the slice representation, by changing $\phi$ to its Fourier transform $\hat{\phi}$, we can rewrite $I_{N}(\phi)$ as a weighted orbital integral of $\hat{\phi}$ whose weight is given by the truncated function $\kappa_N$. Then we change the truncated function (which is given by $\kappa_N$) in the weighted orbital integral to the standard weight factor defined by Arthur. This requires a long and technical argument. But due to the similarity between the model $(G,H)$ and the GR model, this argument will be similar to the GR model case in Section 9 of \cite{Wan15}. After all the arguments above, we can show that
$$I(\phi)=\lim_{N\rightarrow \infty}I_N(\phi)=\sum_{T\in \CT(G)} |W(G,T)|^{-1} \int_{\Ft_0'(F)}D^G(t)^{1/2} \hat{\theta}_{\phi}(t)\ud t$$
where $\CT(G)$ is a set of representatives of maximal tori in $G$, $\Ft'(F)=\Ft(F)\cap \Fg'(F)$, and $\Ft_0'(F)$ is the image of the slice representation which will be an open set of $\Ft'(F)$.

Finally, we just need to apply the standard argument as in the GGP case (Sections 11.4--11.7 of \cite{B15}) to finish the proof of the geometric side of the trace formula.

\subsection{The multiplicity formula}
Let $\pi$ be a smooth admissible (not necessarily irreducible) tempered representation of $G(F)$ with central character $\eta$.
Define the geometric multiplicity $m_{geom}(\pi)$ to be
\begin{align*}
m_{geom}(\pi):=c_{\pi}(1)+\sum_{T\in \CT_{ell}(H_0)}&  |W(H_0,T)|^{-1} \vol(T(F)/Z_G(F))^{-1} \\
&\times\int_{T(F)/Z_G(F)} D^H(t)c_{\pi}(t) \omega^{-1}(t)\ud t.
\end{align*}
Here $c_{\pi}(t)=c_{\theta_{\pi}}(t)$ is the regular germ of $\theta_{\pi}$ at $t$. Note that the expression of $m_{geom}(\pi)$ is almost the same as the geometric expansion $I_{geom}(f)$. The only difference is that we replace the quasi-character $\theta_f$ by $\theta_{\pi}$. The multiplicity formula is just
$$m(\pi)=m_{geom}(\pi).$$
In the next section, we will prove the multiplicity formula by assuming the trace formula holds. Apparently it is enough to prove it for irreducible tempered representations.

\subsection{The reduced models}
In this subsection, we will discuss the reduced models of the pair $(G,H)$. With the notation as in Section \ref{sec:model}, the reduced models are just the models $(G_{\bar{Q}},H_{\bar{Q}})$ where $\bar{Q}=M_Q\bar{U}_Q$ runs over the good parabolic subgroups of $G$. This models will be used in the proof of the spectral side of the trace formula.

We first define the multiplicities for the reduced models. Let $\tau$ be a smooth admissible representation of $G_{\bar{Q}}(F)$ whose central character equals $\eta$ on $Z_G(F)$.
Define the multiplicity $m(\tau)$ to be
$$m(\tau):=\dim(Hom_{H_{\bar{Q}}(F)}(\tau, (\omega\otimes \xi)|_{H_{\bar{Q}}(F)} \otimes \delta_{H_{\bar{Q}}}^{1/2})).$$
Note that as in Proposition \ref{major 5}, when we consider the reduced models, we need to twist the extra modular character $\delta_{H_{\bar{Q}}}^{1/2}$. For simplicity, we will use $\omega_Q\otimes \xi_Q$ (or just $\omega_Q$ if $\xi_Q$ is trivial) to denote the character $(\omega\otimes \xi)|_{H_{\bar{Q}}(F)} \otimes \delta_{H_{\bar{Q}}}^{1/2}$.

For our application, we need to divide the reduced models into two categories. We say the model $(G_{\bar{Q}},H_{\bar{Q}})$ is of  {\it Type I} \label{pg:type-I} if it appears both in the quasi-split case and the non quasi-split case (this is the same situation as in the GR model case, see Appendix B of \cite{Wan15} for details). This is equivalent to say that the Levi subgroup $M_Q(F)$ is isomorphic to $\GL_1(E)\times \GU(J_{4,\varepsilon})(F),\; \GL_2(E)\times \GU(J_{2,\varepsilon})(F)$ or $\GL_1(E)\times \GL_1(E)\times \GU(J_{2,\varepsilon})(F)$. All the rest reduced models are called {\it Type II} models. In particular, Type II models only appear in the quasi-split case.

For the rest part of this section, we will describe the reduce models and the multiplicity formulas associated to them. \textbf{We first consider the Type I reduced models.} If $M_Q(F)$ is isomorphic to $\GL_2(E)\times \GU(J_{2,\varepsilon})(F)$, the reduced model $(G_{\bar{Q}},H_{\bar{Q}})$ is just $(G_0, H_0)$. And the character $\omega_Q$ on $H_{\bar{Q}}$ is just the character $\omega$. In this case, the multiplicity formula is very similar to the $(G,H)$ case. To be specific, given a smooth admissible tempered representation $\tau$ of $G_0(F)$ whose central character equals $\eta$ on $Z_{H_0}(F)$, we define
\begin{align*}
m_{geom}(\tau):=c_{\tau}(1)+\sum_{T\in \CT_{ell}(H_0)} & |W(H_0,T)|^{-1} \vol(T(F)/Z_G(F))^{-1}\\
&\times \int_{T(F)/Z_G(F)} D^{H_0}(t) \theta_{\tau}(t) \omega^{-1}(t) \ud t.
\end{align*}
Then the multiplicity formula is just $m(\tau)=m_{geom}(\tau)$. The two models we get here are the only pure inner forms of each other.

If $M_Q(F)$ is isomorphic to $\GL_1(E)\times \GU(J_{4,\varepsilon})(F)$, the reduced model can be described as follows: $G_{\bar{Q}}(F)=M_{Q}(F)$, $H_{\bar{Q}}=H_{0,\bar{Q}}\ltimes U_{\bar{Q}}$ with
$$H_{0,\bar{Q}}(F)=\{h_Q(a,b)=\begin{pmatrix}a\end{pmatrix}\times diag(b,a,b,b) \colon a,b\in E^{\times},\; N_{E/F}(a)=N_{E/F}(b)\},$$
$$U_{\bar{Q}}(F)=\{u_Q(x,y,z)=\begin{pmatrix}1\end{pmatrix}\times \begin{pmatrix} 1&x&y&z\\0&1&0&\varepsilon^{-1}\bar{x} \\0&0&1&-\bar{y}\\ 0&0&0&1\end{pmatrix}\colon x,y,z\in E,\; z+\bar{z}-\varepsilon^{-1}x\bar{x}+y\bar{y}=0\}.$$
The character $\omega_Q\otimes \xi_Q$ on $H_{\bar{Q}}(F)$ is given by
$$\omega_Q\otimes \xi_Q(h_Q(a,b)u_Q(x,y))=\chi_1(a)\chi_2(b)\psi(y+\bar{y})$$
where $\chi_1$ and $\chi_2$ are some unitary characters of $E^{\times}$ with $\eta=\chi_1\chi_2$. We define the geometric multiplicity to be
$$m_{geom}(\tau):=c_{\tau}(1)+\vol(H_{0,\bar{Q}}(F)/Z_{G_{\bar{Q}}}(F))^{-1}
\int_{Z_{G_{\bar{Q}}}(F)\back H_{0,\bar{Q}}(F)} D^{H_{\bar{Q}}}(t) c_{\tau}(t) \omega_Q(t) \ud t$$
where $\tau$ is any smooth admissible tempered representation of $G_{\bar{Q}}(F)$ whose central character equals $\eta$ on $Z_G(F)$. The two models we get here are the only pure inner forms of each other.

If $M_Q(F)$ is isomorphic to $\GL_1(E)\times \GL_1(E)\times\GU(J_{2,\varepsilon})(F)$, the reduced model can be described as follows: $G_{\bar{Q}}(F)=M_{Q}(F)$, and
$$H_{\bar{Q}}(F)=H_{0,\bar{Q}}(F)=\{h_Q(a,b)=\begin{pmatrix}a\end{pmatrix}\times\begin{pmatrix}b\end{pmatrix}\times diag(a,b) \colon a,b\in E^{\times},\; N_{E/F}(a)=N_{E/F}(b)\}.$$
The character $\omega_Q$ on $H_{\bar{Q}}$ is given by
$$\omega_Q(h_Q(a,b))=\chi_1(a)\chi_2(b)$$
where $\chi_1$ and $\chi_2$ are some unitary characters of $E^{\times}$ with $\eta=\chi_1\chi_2$. We define the geometric multiplicity to be
$$m_{geom}(\tau):=c_{\tau}(1)+\vol(H_{0,\bar{Q}}(F)/Z_{G_{\bar{Q}}}(F))^{-1}\int_{Z_{G_{\bar{Q}}}(F)\back H_{0,\bar{Q}}(F)} D^{H_{\bar{Q}}}(t) c_{\tau}(t) \omega_Q(t) \ud t$$
where $\tau$ is any smooth admissible tempered representation of $G_{\bar{Q}}(F)$ whose central character equals $\eta$ on $Z_G(F)$. The two models we get here are the only pure inner forms of each other.

\begin{rmk}
By the description above, it is easy to see that the extra modular character $\delta_{H_{\bar{Q}}}^{1/2}$ is trivial for all Type I reduced models.
\end{rmk}

\textbf{Then we consider the Type II reduced models.} There are three Type II reduced models (all in the quasi-split case) which correspond to the cases when $M_Q(F)$ is isomorphic to $\GL_3(E)\times \GL_1(F), \GL_2(E)\times \GL_1(E)\times \GL_1(F)$ and $\GL_1(E)\times \GL_1(E)\times \GL_1(E)\times \GL_1(F)$. When $M_Q(F)=\GL_3(E)\times \GL_1(F)$, up to modulo the $\GL_1(F)$-part which is abelian, the reduced model can be described as follows: $G_{\bar{Q}}(F)=\GL_3(E)$ and $H_{\bar{Q}}(F)=H_{0,\bar{Q}}(F)\ltimes U_{\bar{Q}}(F)$ where
\begin{eqnarray*}
H_{0,\bar{Q}}(F)&=& \{\begin{pmatrix}a&0&0\\c&b&0\\0&0&a \end{pmatrix}\colon a,b\in E^{\times}, c\in E,\frac{a}{b}\in F^{\times}, \frac{c}{a}\in \sqrt{\alpha}F\}, \\
U_{\bar{Q}}(F)&=& \{\begin{pmatrix}1&0&x_1\\0&1&x_2\\0&0&1 \end{pmatrix}\colon  x_1,\ x_2\in E\}.
\end{eqnarray*}
And the character $\omega_Q\otimes \xi_Q$ is given by
$$\omega_Q\otimes \xi_Q(\begin{pmatrix}a&0&0\\c&b&0\\0&0&a \end{pmatrix} \begin{pmatrix}1&0&x_1\\0&1&x_2\\0&0&1 \end{pmatrix})=|\frac{a}{b}|^{-3/2}\chi_1(a)\chi_2(b)\psi(x_1)$$
where $\chi_1$ and $\chi_2$ are some unitary characters of $E^{\times}$. Here the factor $|\frac{a}{b}|^{-3/2}$ comes from the extra modular character $\delta_{H_{\bar{Q}}}^{1/2}$.

When $M_{Q}=\GL_2(E)\times \GL_1(E)\times \GL_1(F)$, up to modulo the $\GL_1(F)$ and $\GL_1(E)$ parts which are abelian, the reduced model can be described as follows: $G_{\bar{Q}}(F)=\GL_2(E)$, and
$$H_{\bar{Q}}(F)=H_{0,\bar{Q}}(F)=\{\begin{pmatrix} a&0\\c&b\end{pmatrix}\colon a,b\in E^{\times}, c\in E,\frac{a}{b}\in F^{\times}, \frac{c}{a}\in \sqrt{\alpha}F\}.$$
And the character $\omega_Q$ is given by
$$\omega_Q(\begin{pmatrix} a&0\\c&b\end{pmatrix})=|\frac{a}{b}|^{-1/2}\chi_1(a)\chi_2(b)$$
where $\chi_1$ and $\chi_2$ are some unitary characters of $E^{\times}$, and the factor $|\frac{a}{b}|^{-1/2}$ comes from the extra modular character $\delta_{H_{\bar{Q}}}^{1/2}$.

When $M_Q=\GL_1(E)\times \GL_1(E)\times \GL_1(E)\times \GL_1(F)$, the model is abelian and the multiplicity is trivially equal to $1$ for all irreducible representations (which are just characters).

For all the Type II reduced models, the geometric multiplicity is defined to be
$$m_{geom}(\tau)=c_{\tau}(1).$$
Since $G_{\bar{Q}}$ is a product of the general linear groups for all Type II reduced models, all these models don't have any other pure inner form and each $L$-packet only contains one element.

\begin{rmk}
The most important feature of the Type II reduced models is that there is no elliptic element in $H_{\bar{Q}}(F)$ other than the center. As a result, in the multiplicity formulas for these models, we only have the germ at the identity element. This is an analogue of the Type II reduced models for the GR model (see Appendix B of \cite{Wan15}).
\end{rmk}

The proof of the multiplicity formulas for the reduced models follows from the same, but easier arguments as the proof of the multiplicity formula for the model $(G,H)$. \textbf{Hence by induction, we will assume the multiplicity formulas for all the reduced models hold for the rest part of this paper.}

\subsection{Some consequences of the multiplicity formulas for the reduced models} \label{sec:consequence}
In this subsection, we discuss some consequences of the multiplicity formulas for the reduced models. Let $(G_{\bar{Q}},H_{\bar{Q}})$ be a reduced model and let $\Pi$ be a tempered local Vogan $L$-packet of $G_{\bar{Q}}(F)$ whose central character equals $\eta$ on $Z_G(F)$. If the reduced model is of Type II, $G_{\bar{Q}}$ is the product of some general linear group. Hence there is no other pure inner form of the model $(G_{\bar{Q}},H_{\bar{Q}})$ and the $L$-packet $\Pi$ only contains one element $\tau$. If the reduced model is of Type I, then there is another pure inner form of the model $(G_{\bar{Q}},H_{\bar{Q}})$ (as we described in the previous subsection) and the $L$-packet $\Pi$ contains representations of both groups and may have more than one element.

\begin{thm}\label{main theorem for reduced models}
\begin{enumerate}
\item \label{item:main-thm-1} $\displaystyle\sum_{\tau\in \Pi} m(\tau)=1$.
\item \label{thm:main-thm-item-2} $m(\tau)\leq 1$ for all irreducible tempered representation $\tau$ of $G_{\bar{Q}}(F)$ whose central character equals $\eta$ on $Z_G(F)$.
\end{enumerate}
\end{thm}

\begin{proof}
\eqref{thm:main-thm-item-2} is a direct consequence of \eqref{item:main-thm-1}. For \eqref{item:main-thm-1}, if the reduced model is of Type II, by the discussion above, the $L$-packet $\Pi$ only contains one element $\tau$ which is an irreducible tempered representation of some general linear group. Then we know that $\tau$ is a generic representation. Combining the multiplicity formula in the previous section and the work of Rodier in \cite{Rod81}, we have
$$\sum_{\tau\in \Pi} m(\tau)=m(\tau)=m_{geom}(\tau)=c_{\tau}(1)=1.$$
This proves \eqref{item:main-thm-1} for Type II reduced models.

Next, we consider the Type I reduced models. We will only consider the case when $G_{\bar{Q}}(F)\simeq \GL_2(E)\times \GU(J_{2,\varepsilon})(F)$. The arguments for the rest two cases are similar. For $i=1,2$, fix $\varepsilon_i\in F^{\times}$ with $\eta_{E/F} (\varepsilon_i)=(-1)^{i-1}$ and let $G_{\bar{Q},\varepsilon_i}(F)=\GL_2(E)\times \GU(J_{2,\varepsilon_i})(F)$. Then the $L$-packet $\Pi$ is of the form
$$\Pi=\{\tau_0\otimes \tau_{\varepsilon_1} \colon \tau_{\varepsilon_1}\in \Pi_{\phi}(\GU(J_{2,\varepsilon_1}))\}\cup \{\tau_0\otimes \tau_{\varepsilon_2} \colon \tau_{\varepsilon_2}\in \Pi_{\phi}(\GU(J_{2,\varepsilon_2}))\}$$
where $\tau_0$ is some irreducible tempered representation of $\GL_2(E)$, and $\Pi_{\phi}=\Pi_{\phi}(\GU(J_{2,\varepsilon_1}))\cup \Pi_{\phi} (\GU(J_{2,\varepsilon_2}))$ is a tempered local $L$-packet of $\GU_2(F)$ as in Conjecture \ref{endoscopy}. By the multiplicity formula in the previous subsection, we know that $\sum_{\tau\in \Pi}m(\tau)$ is equal to
\begin{align*}
&\sum_{i=1}^{2} \sum_{\tau_{\varepsilon_i}\in \Pi_{\phi}(\GU(J_{2,\varepsilon_i}))} c_{\tau_0}(1)c_{\tau_{\varepsilon_i}}(1) +\sum_{i=1}^{2}\sum_{\tau_{\varepsilon_i}\in \Pi_{\phi}(\GU(J_{2,\varepsilon_i}))} \sum_{T_i\in \CT_{ell}(\GU(J_{2,\varepsilon_i}))}\\
&\times\nu(T_i) \int_{T_i(F)/Z_{\GU(J_{2,\varepsilon_i})}(F)} D^{\GU(J_{2,\varepsilon_i})}(t_i) \theta_{\tau_0}(t_i) \theta_{\tau_{\varepsilon_i}}(t_i) \omega_{\varepsilon_i}^{-1}(t_i)\ud t_i
\end{align*}
with $\nu(T_i)=|W(\GU(J_{2,\varepsilon_i}),T_i)|^{-1} \vol(T_i(F)/Z_{\GU(J_{2,\varepsilon_i})}(F))^{-1}$. Since $\tau_0$ is a tempered representation of $\GL_2(E)$, it is generic and hence $c_{\tau_0}(1)=1$ by the work of Rodier in \cite{Rod81}. Moreover, by Conjecture \ref{endoscopy} \eqref{item:conjecture-2} together with Rodier's work in \cite{Rod81}, we have
$$\sum_{i=1}^{2} \sum_{\tau_{\varepsilon_i}\in \Pi_{\phi}(\GU(J_{2,\varepsilon_i}))} c_{\tau_0}(1)c_{\tau_{\varepsilon_i}}(1)= \sum_{i=1}^{2} \sum_{\tau_{\varepsilon_i}\in \Pi_{\phi}(\GU(J_{2,\varepsilon_i}))} c_{\tau_{\varepsilon_i}}(1)=1.$$
Hence $\sum_{\tau\in \Pi}m(\tau)$ equals
\begin{align*}
1+&\sum_{i=1}^{2}\sum_{T_i\in \CT_{ell}(\GU(J_{2,\varepsilon_i}))}  \nu(T_i) \\
 &\times\int_{T_i(F)/Z_{\GU(J_{2,\varepsilon_i})}(F)} D^{\GU(J_{2,\varepsilon_i})}(t_i) \theta_{\tau_0}(t_i) \theta_{\Pi_{\phi}(\GU(J_{2,\varepsilon_i}))}(t_i) \omega_{\varepsilon_i}^{-1}(t_i) \ud t_i.
\end{align*}
Here we recall from Conjecture \ref{endoscopy} that $\theta_{\Pi_{\phi}(\GU(J_{2,\varepsilon_i}))}=\sum_{\tau_{\varepsilon_i}\in \Pi_{\phi} (\GU(J_{2,\varepsilon_i}))} \theta_{\tau_{\varepsilon_i}}$.

Now we are ready to prove the theorem. We have a natural bijection
$$T_1\in \CT_{ell}(\GU(J_{2,\varepsilon_1}))\leftrightarrow T_2\in \CT_{ell}(\GU(J_{2,\varepsilon_2}))$$
between the maximal elliptic tori of $\GU(J_{2,\varepsilon_1})(F)$ and $\GU(J_{2,\varepsilon_2})(F)$. Hence in order to prove the theorem, it is enough to show that for all $T_1\leftrightarrow T_2$, we have
\begin{align*}
&\nu(T_1) \int_{T_1(F)/Z_{\GU(J_{2,\varepsilon_1})}(F)} D^{\GU(J_{2,\varepsilon_1})}(t_1) \theta_{\tau_0}(t_1) \theta_{\Pi_{\phi} (\GU(J_{2,\varepsilon_1}))}(t_1) \omega_{\varepsilon_1}^{-1}(t_1)\ud t_1\\
=&-\nu(T_2) \int_{T_2(F)/Z_{\GU(J_{2,\varepsilon_2})}(F)} D^{\GU(J_{2,\varepsilon_2})}(t_2) \theta_{\tau_0}(t_2) \theta_{\Pi_{\phi} (\GU(J_{2,\varepsilon_2}))}(t_2) \omega_{\varepsilon_2}^{-1}(t_2)\ud t_2.
\end{align*}
We fix such a pair $(T_1,T_2)$. It is easy to see that $\nu(T_1)=\nu(T_2)$. For $t_1\in T_1(F)$ and $t_2\in T_2(F)$, we write $t_1\leftrightarrow t_2$ if they have the same characteristic polynomial. Then it is enough to show that for all $t_1\in T_1(F)_{reg}$ and $t_2\in T_2(F)_{reg}$ with $t_1\leftrightarrow t_2$, we have
\begin{align}
&D^{\GU(J_{2,\varepsilon_1})}(t_1) \theta_{\tau_0}(t_1) \theta_{\Pi_{\phi}(\GU(J_{2,\varepsilon_1}))}(t_1) \omega_{\varepsilon_1}^{-1}(t_1) \nonumber \\
=& -D^{\GU(J_{2,\varepsilon_2})}(t_2) \theta_{\tau_0}(t_2) \theta_{\Pi_{\phi}(\GU(J_{2,\varepsilon_2}))}(t_2) \omega_{\varepsilon_2}^{-1}(t_2). \label{6.1}
\end{align}
Since $t_1\leftrightarrow t_2$, we have
$$D^{\GU(J_{2,\varepsilon_1})}(t_1) \theta_{\tau_0}(t_1)  \omega_{\varepsilon_1}^{-1}(t_1)= D^{\GU(J_{2,\varepsilon_2})}(t_2) \theta_{\tau_0}(t_2) \omega_{\varepsilon_2}^{-1}(t_2).$$
By Conjecture \ref{endoscopy} \eqref{item:conjecture-3}, we also have $\theta_{\Pi_{\phi}(\GU(J_{2,\varepsilon_1}))}(t_1)=-\theta_{\Pi_{\phi}(\GU(J_{2,\varepsilon_2}))} (t_2)$. This proves \eqref{6.1} and completes the proof of the theorem.
\end{proof}

\noindent
The following proposition will be proved in Appendix \ref{sec:appendix} by the orbit method.
\begin{prop}\label{Gelfand pair}
Let $\bar{Q}=M_Q\bar{U}_{Q}$ be a good parabolic subgroup of $G$, and let $\tau$ be an admissible tempered representation of $M_{Q}(F)$ whose central character equals $\eta$ on $Z_G(F)$. Set $\pi=I_{\bar{Q}}^{G}(\tau)$. Then
$$m(\pi)\leq m(\tau).$$
\end{prop}

\begin{rmk}\label{irreducible}
In order to prove Proposition \ref{Gelfand pair}, it is enough to consider the case when $\tau$ is irreducible.
\end{rmk}

\begin{cor}\label{Gelfand pair cor}
For all $\pi\in \Pi_{temp}(G,\eta)\smallsetminus\Pi_2(G,\eta)$, we have
$$m(\pi)\leq 1.$$
\end{cor}

\begin{proof}
Since $\pi\in \Pi_{temp}(G,\eta)\smallsetminus\Pi_2(G,\eta)$, we can find a good parabolic subgroup $\bar{Q}=M_Q\bar{N}_{Q}$ of $G$ and an irreducible tempered representation $\tau$ of $M_Q(F)$ such that $\pi=I_{\bar{Q}}^{G}(\tau)$. By Theorem \ref{main theorem for reduced models} (2) and Proposition \ref{Gelfand pair}, we have
$$m(\pi)= m(I_{\bar{Q}}^{G}(\tau))\leq m(\tau)\leq 1.$$
This proves the Corollary.
\end{proof}

The following theorem is a stronger version of Proposition \ref{Gelfand pair}. It will be proved in Section \ref{sec:intertwining} by assuming Proposition \ref{Gelfand pair} holds.
\begin{thm}\label{multiplcity parabolic induction}
With the same assumptions and notation as in Proposition \ref{Gelfand pair}, we have $m(\pi)= m(\tau).$
\end{thm}

\section{The proof of the main theorems} \label{sec:proof-main}
In this section, we are going to prove our main theorems by assuming the trace formula in Theorem \ref{trace formula} holds. We will prove the  unitary similitude group case in Section \ref{sec:GU}, and the unitary group case in Section \ref{sec:unitary}.

\subsection{The unitary similitude  group case} \label{sec:GU}
We first prove the multiplicity formula $m(\pi)=m_{geom}(\pi)$ for all tempered representations. For simplicity, we still use $(G,H)$ to denote $(G_{\varepsilon},H_{\varepsilon})$ and $\omega\otimes \xi$ to denote $\omega_{\varepsilon}\otimes \xi_{\varepsilon}$. We need a proposition.

\begin{prop}
Let $\bar{Q}=M_Q\bar{U}_{Q}$ be a good parabolic subgroup of $G$, and let $\tau$ be an irreducible tempered representation $M_{Q}(F)$ whose central character equals $\eta$ on $Z_G(F)$. Set $\pi=I_{\bar{Q}}^{G}(\tau)$. Then
\begin{equation}\label{multiplcity parabolic induction 1}
m(\pi)=m(\tau) \text{ and } m_{geom}(\pi)=m_{geom}(\tau).
\end{equation}
\end{prop}

\begin{proof}
The first equality follows from Theorem \ref{multiplcity parabolic induction}, so it is enough to prove $m_{geom}(\pi)=m_{geom}(\tau)$. This is a direct consequence of Lemma 2.3 of \cite{W12} together with the definitions of $m_{geom}(\pi)$ and $m_{geom}(\tau)$.
\end{proof}

Now the multiplicity formula will be a direct consequence of the trace formula in Theorem \ref{trace formula}, together with \eqref{multiplcity parabolic induction 1}. The argument is the same as the GGP case (Proposition 11.3.1 of \cite{B15}) and we will skip it here.

We are ready to prove Theorem \ref{main GU}. The proof is very similar to the proof of Theorem \ref{main theorem for reduced models} for the reduced model cases, so we only give a sketch of it. For $i=1,2$, fix $\varepsilon_i\in F^{\times}$ with $\eta_{E/F}(\varepsilon_i)=(-1)^{i-1}$ as before. Let $\Pi_{\phi}$ be a local tempered Vogan $L$-packet of $\GU_6$. Then we can write $\Pi_{\phi}$ as $\Pi_{\phi}=\Pi_{\phi}(G_{\varepsilon_1})\cup \Pi_{\phi}(G_{\varepsilon_2})$. Our goal is to show that
\begin{equation}\label{5.1}
\sum_{\pi\in \Pi_{\phi}}m(\pi)=\sum_{i=1}^{2}\sum_{\pi_{\varepsilon_i}\in \Pi_{\phi}(G_{\varepsilon_i})} m(\pi_{\varepsilon_i}) =1.
\end{equation}

By the multiplicity formula, we have
\begin{align*}
\sum_{\pi\in \Pi_{\phi}}m(\pi)=&\sum_{i=1}^{2}\sum_{\pi_{\varepsilon_i}\in \Pi_{\phi}(G_{\varepsilon_i})} \left( c_{\pi_{\varepsilon_i}}(1)+\sum_{T_i\in \CT_{ell}(H_{0,\varepsilon_i})}   |W(H_{0,\varepsilon_i},T_i)|^{-1} \right. \\
& \left. \vol(T_i(F)/Z_{G_{\varepsilon_i}}(F))^{-1}\int_{T_i(F)/Z_{G_{\varepsilon_i}}(F)} D^{H_{\varepsilon_i}}(t_i)c_{\pi_{\varepsilon_i}}(t_i) \omega_{\varepsilon_i}^{-1}(t_i)\ud t_i \right).
\end{align*}
By the same argument as in the proof of Theorem \ref{main theorem for reduced models}, together with Conjecture \ref{endoscopy} \eqref{item:conjecture-2} and \eqref{item:conjecture-3}, we can show that
\begin{align*}
 &\sum_{i=1}^{2}\sum_{\pi_{\varepsilon_i}\in \Pi_{\phi}(G_{\varepsilon_i})} c_{\pi_{\varepsilon_i}}(1)=1,\\
&\sum_{i=1}^{2}\sum_{\pi_{\varepsilon_i}\in \Pi_{\phi}(G_{\varepsilon_i})} \sum_{T_i\in \CT_{ell}(H_{0,\varepsilon_i})}  \nu(T_i)\int_{T_i(F)/Z_{G_{\varepsilon_i}}(F)} D^{H_{\varepsilon_i}}(t_i)c_{\pi_{\varepsilon_i}}(t_i) \omega_{\varepsilon_i}^{-1}(t_i)\ud t_i=0
\end{align*}
with $\nu(T_i)=|W(H_{0,\varepsilon_i},T_i)|^{-1} \vol(T_i(F)/Z_{G_{\varepsilon_i}}(F))^{-1}$. This proves \eqref{5.1} and finishes the proof of Theorem \ref{main GU}.

\subsection{The unitary group case} \label{sec:unitary}

In this subsection, we are going to prove Theorem \ref{main U}. The idea is to study the relations between the models associated to  unitary similitude groups and the models associated to unitary groups. We first recall the definition of the character $\omega_{\varepsilon}$ (resp. $\omega_{1,\varepsilon}$) of $H_{0,\varepsilon}(F)$ (resp. $H_{0,1,\varepsilon}(F)$) in Section \ref{sec:introduction}. We recall
\begin{align*}
&\omega_{\varepsilon}(m(h,h))=\chi_E(\det(h))\chi_F(\lambda(h)) \text{ for } h\in \GU(J_{2,\varepsilon})(F),\\
&\omega_{1,\varepsilon}(m(h_1,h_1))=\chi_E(\det(h_1)) \text{ for } h_1\in \mathrm{U}(J_{2,\varepsilon})(F),
\end{align*}
where $\chi_E$ (resp. $\chi_F$) is a character of $E^{\times}$ (resp. $F^{\times}$), $\lambda$ is the similitude character, $m(h,h)=diag(h,h,\lambda(h)w_2{}^t\bar{h}^{-1}w_2)\in H_{0,\varepsilon}(F)$ and $m(h_1,h_1)=diag(h_1,h_1,w_2{}^t\bar{h}_{1}^{-1}w_2)\in H_{0,1,\varepsilon}(F)$. Let $\eta_{E/F}:F^{\times}\rightarrow \BC^{\times}$ be the quadratic character associated to $E$ as before. We define a new character $\omega_{\varepsilon}'$ of $H_{0,\varepsilon}(F)$ to be
$$\omega_{\varepsilon}'(m(h,h))=\chi_E(\det(h))\chi_F(\lambda(h))\eta_{E/F}(\lambda(h)),\;h\in \GU(J_{2,\varepsilon})(F).$$
And for any smooth admissible representation $\pi_{\varepsilon}$ of $G_{\varepsilon}(F)$ with central character $\eta$ (note that $\omega_{\varepsilon}$ is equal to $\omega_{\varepsilon}'$ on the center of $G_{\varepsilon}(F)$), we define the multiplicity
$$m(\pi_{\varepsilon})'=\dim(\Hom_{H_{\varepsilon}(F)} (\pi_{\varepsilon},\omega_{\varepsilon}'\otimes \xi_{\varepsilon})).$$
The next proposition is essential in the proof of Theorem \ref{main U}.

\begin{prop}\label{GU v.s. U}
For any irreducible smooth representation $\pi_{\varepsilon}$ of $G_{\varepsilon}(F)=\GU(J_{6,\varepsilon})(F)$ with central character $\eta$, let $\pi_{1,\varepsilon}$ be the restriction of $\pi_{\varepsilon}$ to $G_{1,\varepsilon}(F)=\mathrm{U}(J_{6,\varepsilon})(F)$ which is a smooth admissible representation (not necessarily irreducible) of $G_{1,\varepsilon}(F)$ with central character $\eta_1$. Then
$$m(\pi_{1,\varepsilon})=m(\pi_{\varepsilon})+m(\pi_{\varepsilon})'.$$
\end{prop}

\begin{proof}
Fix an element $B_{\varepsilon}\in \GU(J_{2,\varepsilon})(F)$ such that $\delta=\lambda(B_{\varepsilon})\notin \Im(N_{E/F})$. Define $h_{\varepsilon}=m(B_{\varepsilon},B_{\varepsilon})\in H_{0,\varepsilon}(F)\subset G_{\varepsilon}(F)$. Then
\begin{equation}\label{GU v.s. U 1}
\begin{array}{l}
G_{\varepsilon}(F)=G_{1,\varepsilon}(F)Z_{G_{\varepsilon}}(F)\cup h_{\varepsilon}\cdot G_{1,\varepsilon}(F)Z_{G_{1,\varepsilon}}(F),\\
H_{\varepsilon}(F)=H_{1,\varepsilon}(F)Z_{H_{\varepsilon}}(F)\cup h_{\varepsilon}\cdot H_{1,\varepsilon}(F)Z_{H_{1,\varepsilon}}(F).
\end{array}
\end{equation}
In terms of the characters, we also have
\begin{equation}\label{GU v.s. U 2}
\begin{array}{l}
\omega_{\varepsilon}\otimes \xi_{\varepsilon}|_{H_{1,\varepsilon}(F)}=\omega_{\varepsilon}'\otimes \xi_{\varepsilon}|_{H_{1,\varepsilon}(F)}= \omega_{1,\varepsilon}\otimes \xi_{1,\varepsilon},\\
\omega_{\varepsilon}(h_{\varepsilon})=-\omega_{\varepsilon}'(h_{\varepsilon})=a
\end{array}
\end{equation}
where $a\in \BC^{\times}$ is a nonzero complex number. We have the following two cases.

\textbf{Case 1:} Assume that $\pi_{1,\varepsilon}$ is irreducible. We first prove that
\begin{equation}\label{GU v.s. U 3}
m(\pi_{1,\varepsilon})\leq 2.
\end{equation}
If not, then we can find at least three linearly independent elements $l_1,l_2,l_3\in \Hom_{H_{1,\varepsilon}(F)}(\pi_{1,\varepsilon}, \omega_{1,\varepsilon} \otimes \xi_{1,\varepsilon})$. Then $l_i\circ \pi_{\varepsilon}(h_{\varepsilon})$ are also linearly independent elements in $\Hom_{H_{1,\varepsilon}(F)}(\pi_{1,\varepsilon}, \omega_{1,\varepsilon} \otimes \xi_{1,\varepsilon})$. If for $1\leq i\leq 3$, $c_i\cdot l_i= l_i\circ \pi_{\varepsilon}(h_{\varepsilon})$ for $c_i=\pm a$, then $l_i\in \Hom_{H_{\varepsilon}(F)}(\pi_{\varepsilon},\omega_{\varepsilon}\otimes \xi_{\varepsilon})$ when $c_i=a$, and $l_i\in \Hom_{H_{\varepsilon}(F)}(\pi_{\varepsilon},\omega_{\varepsilon}'\otimes \xi_{\varepsilon})$ when $c_i=-a$ (this is due to \eqref{GU v.s. U 1} and \eqref{GU v.s. U 2}). As $m(\pi_{\varepsilon}),m(\pi_{\varepsilon})'\leq 1$ by Theorem \ref{main GU}, there exists $1\leq i\leq 3$, such that $l_i\neq \pm a\cdot l_i\circ \pi_{\varepsilon}(h_{\varepsilon})$. Without loss of generality, we assume that $i=1$. Then we can find $v\in \pi_{1,\varepsilon}$ such that
$$l_1(v)=1 \text{ and } l_1(v')\neq \pm a \text{ where } v'=\pi_{\varepsilon}(h_{\varepsilon})v.$$

As $l_1,l_2$ and $l_3$ are linearly independent, up to change $l_2$ and $l_3$, we may assume that
$$l_2(v)=l_3(v)=l_3(v')=0.$$
For $1\leq i\leq 3$, we define elements $T_i\in  \Hom_{H_{\varepsilon}(F)}(\pi_{\varepsilon},\omega_{\varepsilon}\otimes \xi_{\varepsilon})$ and $T_i'\in \Hom_{H_{\varepsilon}(F)}(\pi_{\varepsilon},\omega_{\varepsilon}'\otimes \xi_{\varepsilon})$ to be
\begin{equation}\label{eq:T-T'}
T_i=a\cdot l_i+l_i\circ \pi_{\varepsilon}(h_{\varepsilon})\text{ and } T_i'=a\cdot l_i-l_i\circ \pi_{\varepsilon}(h_{\varepsilon}).
\end{equation}
Then we know $T_1,T_1'\neq 0$. On the other hand, since $l_3\neq 0$, at least one of $T_3$ and $T_3'$ is nonzero. Without loss of generality, we may assume that $T_3\neq 0$. Since $T_1(v)=a\cdot l_1(v)+l_1(v')\neq 0$ and $T_3(v)=a\cdot l_3(v)+l_3(v')=0$, $T_1$ and $T_3$ are linearly independent which implies that $m(\pi_{\varepsilon})\geq 2$. We get a contradiction and this proves \eqref{GU v.s. U 3}. Now we are ready to prove the proposition for this case. There are four subcases.

\textbf{Case 1(a):} If $m(\pi_{\varepsilon})=m(\pi_{\varepsilon})'=1$, by \eqref{GU v.s. U 3}, it is enough to show that $m(\pi_{1,\varepsilon})\geq 2$. Choose nonzero linear functionals $l\in \Hom_{H_{\varepsilon}(F)}(\pi_{\varepsilon},\omega_{\varepsilon}\otimes \xi_{\varepsilon})$ and $l'\in \Hom_{H_{\varepsilon}(F)}(\pi_{\varepsilon},\omega_{\varepsilon}'\otimes \xi_{\varepsilon})$. Then $l$ and $l'$ are linearly independent since the characters $\omega_{\varepsilon}$ and $\omega_{\varepsilon}'$ are different. But we also have $l,l'\in \Hom_{H_{1,\varepsilon}(F)}(\pi_{1, \varepsilon}, \omega_{1,\varepsilon} \otimes \xi_{1,\varepsilon})$ by \eqref{GU v.s. U 2}. This implies that $m(\pi_{1,\varepsilon})\geq 2$.

\textbf{Case 1(b):} If $m(\pi_{\varepsilon})=0$ and $m(\pi_{\varepsilon})'=1$, we choose a nonzero linear functional $l'\in \Hom_{H_{\varepsilon}(F)}(\pi_{\varepsilon},\omega_ {\varepsilon}'\otimes \xi_{\varepsilon})$. We have $l'\in \Hom_{H_{1,\varepsilon}(F)}(\pi_{1, \varepsilon}, \omega_{1,\varepsilon} \otimes \xi_{1,\varepsilon})$ which implies that $m(\pi_{1,\varepsilon})\geq 1$. Hence it is enough to show that $m(\pi_{1,\varepsilon})\leq 1$. If not, we can choose $l\in \Hom_{H_{1,\varepsilon}(F)}(\pi_{1, \varepsilon}, \omega_{1,\varepsilon} \otimes \xi_{1,\varepsilon})$ that is linearly independent with $l'$. As in the discussion above, we define elements $T\in  \Hom_{H_{\varepsilon}(F)}(\pi_{\varepsilon},\omega_{\varepsilon}\otimes \xi_{\varepsilon})$ and $T'\in \Hom_{H_{\varepsilon}(F)}(\pi_{\varepsilon},\omega_{\varepsilon}'\otimes \xi_{\varepsilon})$ as in \eqref{eq:T-T'}.

Since $m(\pi_{\varepsilon})=0$, we have $T=0$ and hence $T'=2a\cdot l$ which is linearly independent with $l'$. This implies that $m(\pi_{\varepsilon})' \geq 2$, a contradiction.  This proves the proposition for this case.

\textbf{Case 1(c):} If $m(\pi_{\varepsilon})=1$ and $m(\pi_{\varepsilon})'=0$, the argument is similar to the previous case and we will skip it here.

\textbf{Case 1(d):} If $m(\pi_{\varepsilon})=m(\pi_{\varepsilon})'=0$, it is enough to show that $m(\pi_{1,\varepsilon})=0$. If not, choose a nonzero linear functional $l\in \Hom_{H_{1,\varepsilon}(F)}(\pi_{1, \varepsilon}, \omega_{1,\varepsilon} \otimes \xi_{1,\varepsilon})$. As in the previous cases, we define elements $T\in  \Hom_{H_{\varepsilon}(F)}(\pi_{\varepsilon},\omega_{\varepsilon}\otimes \xi_{\varepsilon})$ and $T'\in \Hom_{H_{\varepsilon}(F)}(\pi_{\varepsilon},\omega_{\varepsilon}'\otimes \xi_{\varepsilon})$ as in \eqref{eq:T-T'}.
Since $m(\pi_{\varepsilon})=m(\pi_{\varepsilon})'=0$, we have $T=T'=0$ which implies that $l=0$. We get a contradiction and this proves the proposition for Case 1.

\textbf{Case 2:} Assume that $\pi_{1,\varepsilon}$ is reducible. By Lemma \ref{GU to U 1}, $\pi_{1,\varepsilon}=\pi_1\oplus \pi_2$. Moreover, as a vector space, we have $\pi_{\varepsilon}(h_{\varepsilon})\pi_1= \pi_2$ and $\pi_{\varepsilon}(h_{\varepsilon})\pi_2=\pi_1$. In order to prove the proposition, we only need to show that $m(\pi_{1,\varepsilon})=m(\pi_1)+m(\pi_2)=2m(\pi_{\varepsilon})=2m(\pi_{\varepsilon})'$ where $m(\pi_j)=\dim(\Hom_{H_{1,\varepsilon}(F)}(\pi_j, \omega_{1,\varepsilon} \otimes \xi_{1,\varepsilon}))$ for $j=1,2$. We only prove the identity $m(\pi_1)+m(\pi_2)=2m(\pi_{\varepsilon})$, the proof of the other one (i.e. $m(\pi_1)+m(\pi_2)=2m(\pi_{\varepsilon})'$) is similar.

It is easy to see that the map
$$l_1\in \Hom_{H_{1,\varepsilon}(F)}(\pi_1, \omega_{1,\varepsilon} \otimes \xi_{1,\varepsilon}) \mapsto l_2:=l_1\circ \pi_{\varepsilon}( h_{\varepsilon}) \in \Hom_{H_{1,\varepsilon}(F)}(\pi_2, \omega_{1,\varepsilon} \otimes \xi_{1,\varepsilon})$$
is an isomorphism. This implies that $m(\pi_1)=m(\pi_2)$. So it is enough to show that $m(\pi_1)=m(\pi_{\varepsilon})$. There are two subcases.

\textbf{Case 2(a):} If $m(\pi_{\varepsilon})=0$, we need to show that $m(\pi_1)=0$. If not, choose a nonzero element $l_1\in \Hom_{H_{1,\varepsilon} (F)}(\pi_1, \omega_{1,\varepsilon} \otimes \xi_{1,\varepsilon})$ and let $l_2=l_1\circ \pi_{\varepsilon}( h_{\varepsilon})$ be an element in $\Hom_{H_{1,\varepsilon}(F)}(\pi_2, \omega_{1,\varepsilon} \otimes \xi_{1,\varepsilon})$. Then it is easy to see that $l=a\cdot l_1 + l_2$ is a nonzero element in $\Hom_{H_{\varepsilon}(F)}(\pi_{\varepsilon},\omega_{\varepsilon}\otimes \xi_{\varepsilon})$ which is a contradiction. This proves the proposition for this case.

\textbf{Case 2(b):} If $m(\pi_{\varepsilon})=1$, choose a nonzero element $l\in \Hom_{H_{\varepsilon}(F)}(\pi_{\varepsilon}, \omega_{\varepsilon} \otimes \xi_{\varepsilon})$ and let $l_1=l|_{\pi_1}$. Then $l_1\in \Hom_{H_{1,\varepsilon}(F)}(\pi_1, \omega_{1,\varepsilon} \otimes \xi_{1,\varepsilon})$. If $l_1=0$, for all $w\in \pi_2$, we have $l(w)=a^{-1}\cdot l(\pi_{\varepsilon}( h_{\varepsilon})w)=a^{-1}\cdot l_1(\pi_{\varepsilon}( h_{\varepsilon})w)=0$. This implies that $l=0$ which is a contradiction. Hence $l_1\neq 0$ which implies that $m(\pi_1)\geq 1$. So it remains to show that $m(\pi_1)\leq 1$. If not, choose an element $l_1'\in \Hom_{H_{1,\varepsilon}(F)}(\pi_1, \omega_{1,\varepsilon} \otimes \xi_{1,\varepsilon})$ that is linearly independent with $l_1$. Set $l'=a\cdot l_1'+ l_2'$ with $l_2'=l_1'\circ \pi_{\varepsilon}( h_{\varepsilon})$. Then $l'\in \Hom_{H_{\varepsilon}(F)}(\pi_{\varepsilon},\omega_{\varepsilon}\otimes \xi_{\varepsilon})$ and $l,l'$ are linearly independent. This implies that $m(\pi_{\varepsilon})\geq 2$ which is a contradiction. This completes the proof of the proposition.
\end{proof}

Then we prove a multiplicity formula for the model $(G_{1,\varepsilon},H_{1,\varepsilon})$. To simplify the notation, we will omit the subscript $\varepsilon$. We start with a lemma.
\begin{lem}\label{5.4}
Let $\pi$ be irreducible tempered representation of $G(F)$ with central character $\eta$, and let $\pi_1$ be the restriction of $\pi$ to $G_1(F)$. Then we have $m(\pi_1)=m_{geom}(\pi_1)$ where the geometric multiplicity $m_{geom}(\pi_1)$ is defined to be
\begin{align*}
m_{geom}(\pi_1):=2c_{\pi_1}(1)+&\sum_{T_1\in \CT_{ell}(H_{0,1})} |W(H_{0,1},T_1)|^{-1}\\
&\times \int_{T_1(F)/Z_{G_1}(F)} D^{H_1}(t_1) c_{\pi_1}(t_1) \omega_1(t_1)^{-1} \ud_vt_1.
\end{align*}
Here $c_{\pi_1}(t)=c_{\theta_{\pi_1}}(t)$ is the regular germ of $\theta_{\pi_1}$ at $t$ and ${\rm d}_vt_1$ is the Haar measure on $T_1(F)$ such that the volume of $T_1(F)/Z_{G_1}(F)$ is equal to $1$.
\end{lem}

\begin{proof}
Combining Proposition \ref{GU v.s. U} with the multiplicity formula for the  unitary similitude group case, we have
\begin{align}
m(\pi_1)=2c_{\pi}(1)+&\sum_{T\in \CT_{ell}(H_0)} |W(H_0,T)|^{-1}\nonumber \\
&\times\int_{T(F)/Z_G(F)} D^{H}(t) c_{\pi}(t)\omega(t)^{-1} (1+\eta_{E/F}(\lambda(t))) \ud_v t \label{5.2}
\end{align}
where $d_vt$ is the Haar measure on $T(F)$ such that the volume of $T(F)/Z_G(F)$ is equal to $1$. For each $T\in \CT_{ell}(H_0)$, let $T_1=T\cap H_{0,1}$ which is a maximal elliptic torus of $H_{0,1}$. There is a bijection between the set $\CT_{ell}(H_0)$ and the set of quadratic extensions of $F$. If $T$ corresponds to a quadratic extension other than $E$, then there exist an element $\gamma\in T(F)$ such that $T(F)=T_1(F)Z_G(F)\cup \gamma T_1(F)Z_G(F)$ and $\ker(\eta_{E/F}\circ \lambda|_{T})=T_1(F)Z_G(F)$. Hence we have
\begin{align*}
 & \int_{T(F)/Z_G(F)} D^{H}(t) c_{\pi}(t)\omega(t)^{-1} (1+\eta_{E/F}(\lambda(t))) \ud_vt\\
=&2\int_{T(F)/Z_G(F)}D^{H}(t)c_{\pi}(t)\omega(t)^{-1} 1_{T_1(F)Z_G(F)}(t)\ud_vt\\
=&\int_{T_1(F)/Z_{G_1(F)}} D^{H}(t_1) c_{\pi}(t_1)\omega(t_1)^{-1} \ud_vt_1
\end{align*}
where $d_v t_1$ is the Haar measure on $T_1(F)$ such that the volume of $T_1(F)/Z_{G_1}(F)$ is equal to $1$. If $T$ corresponds to $E$, then $\eta_{E/F}\circ \lambda$ is trivial on $T(F)$ and $T(F)=T_1(F)Z_G(F)$. Hence we have
\begin{align*}
 &\int_{T(F)/Z_G(F)} D^{H}(t) c_{\pi}(t)\omega(t)^{-1} (1+\eta_{E/F}(\lambda(t))) \ud_vt\\
 =&2\int_{T_1(F)/Z_{G_1(F)}} D^{H}(t_1) c_{\pi}(t_1)\omega(t_1)^{-1} \ud_vt_1
\end{align*}
where ${\rm d}_v t_1$ is the Haar measure on $T_1(F)$ such that the volume of $T_1(F)/Z_{G_1}(F)$ is equal to $1$. Combining the above two equations with \eqref{5.2}, we have
\begin{align*}
m(\pi_1)=2c_{\pi}(1)+&\sum_{T\in \CT_{ell}(H_0)} |W(H_0,T)|^{-1} \mu(T)\\
&\times \int_{T_1(F)/Z_{G_1}(F)} D^{H}(t_1) c_{\pi}(t_1)\omega(t_1)^{-1} \ud_vt_1
\end{align*}
where $\mu(T)$ is equal to 2 if $T$ corresponds to $E$, and equal to  1 otherwise.

Since $\pi_1$ is the restriction of $\pi$ to $G_1(F)$, we have
$$c_{\pi}(1)=c_{\pi_1}(1)\text{ and } c_{\pi}(t_1)=c_{\pi_1}(t_1), \text{ for all } t_1\in T_1(F)_{reg}.$$
Moreover it is easy to see from the definition that
$$D^{H}(t_1)=D^{H_1}(t_1),\omega(t_1)=\omega_1(t_1),\;\forall t_1\in T_1(F).$$
Hence we have
\begin{align}
m(\pi_1)=2c_{\pi_1}(1)+&\sum_{T\in \CT_{ell}(H_0)} |W(H_{0},T)|^{-1} \mu(T) \nonumber\\
&\times\int_{T_1(F)/Z_{G_1}(F)} D^{H_1}(t_1) c_{\pi_1}(t_1) \omega_1(t_1)^{-1} \ud_vt_1.\label{5.3}
\end{align}
To end the proof, we only need to describe the relations between $\CT_{ell}(H_0)$ and $\CT_{ell}(H_{0,1})$. There are two cases. When $\eta_{E/F}(-\varepsilon)=1$, we have a surjective map from $\CT_{ell}(H_{0,1})$ to the set of all the quadratic extensions of $F$. The fiber of this map has two elements at $E$, and has one element at all the other quadratic extensions. If $T_1'\in \CT_{ell}(H_{0,1})$ maps to a quadratic extension other than $E$, let $T\in \CT_{ell}(H_{0})$ corresponds to the same quadratic extension. Then $T_1=T\cap H_{0,1}$ is a maximal elliptic torus of $H_{0,1}$ which is conjugated to $T_1'$ and we also have $|W(H_0,T)|=|W(H_{0,1},T_1)|$. On the other hand, if $T_1',T_1''\in \CT_{ell}(H_{0,1})$ are the fiber of $E$, we can choose two maximal elliptic tori $T$ and $T_0$ of $H_0$ such that both $T$ and $T_0$ correspond to $E$ and $T_1=T\cap H_{0,1}$ (resp. $T_{1,0}=T_0\cap H_{0,1}$) is a maximal elliptic torus of $H_{0,1}$ which is conjugated to $T_1'$ (resp. $T_1''$). Moreover, in this case, we also have $|W(H_0,T)|=|W(H_{0,1},T_1)|$ and $|W(H_0,T_0)|=|W(H_{0,1},T_{1,0})|$. The upshot is that $T$ and $T_0$ are conjugated to each other in $H_0$, but not in $H_{0,1}$. Then the lemma just follows from \eqref{5.3}.

Now if $\eta_{E/F}(-\varepsilon)=-1$, then we have a bijection map from $\CT_{ell}(H_{0,1})$ to the set of all the quadratic extensions of $F$. If $T_1'\in \CT_{ell}(H_{0,1})$ maps to a quadratic extension other than $E$, let $T\in \CT_{ell}(H_{0})$ corresponds to the same quadratic extension. Then $T_1=T\cap H_{0,1}$ is a maximal elliptic torus of $H_{0,1}$ which is conjugated to $T_1'$ and we also have $|W(H_0,T)|=|W(H_{0,1},T_1)|$. On the other hand, if $T_1'\in \CT_{ell}(H_{0,1})$ maps to the quadratic extension $E$, let $T\in \CT_{ell}(H_{0})$ corresponds to the quadratic extension $E$. Then $T_1=T\cap H_{0,1}$ is a maximal elliptic torus of $H_{0,1}$ which is conjugated to $T_1'$. Moreover, in this case, we have $|W(H_0,T)|=2$ and $|W(H_{0,1},T_1)|=1$. Then the lemma just follows from \eqref{5.3}.
\end{proof}

\begin{prop}\label{multiplicity formula unitary}
Let $\pi_1$ be an irreducible admissible tempered representation of $G_1(F)$ with central character $\eta_1$. Then $m(\pi_1)=m_{geom}(\pi_1)$ where the geometric multiplicity $m_{geom}(\pi_1)$ is defined in the previous lemma.
\end{prop}

\begin{proof}
By Lemmas \ref{GU to U 2} and \ref{GU to U 3}, there exists an irreducible admissible tempered representation $\pi$ of $G(F)$ with central character $\eta$ such that $\pi_1$ is a direct summand of $\pi|_{G_1}$. If $\pi_1=\pi|_{G_1}$, then the multiplicity formula just follows from Lemma \ref{5.4}. If not, by Lemma \ref{GU to U 1}, we have $\pi|_{G_1}=\pi_1\oplus \pi_2$ with $\pi_2=\pi_1\circ Ad(g)$ is another irreducible tempered representation of $G_1(F)$ with central character $\eta_1$. Here $g$ is an element in $G(F)$ with $\lambda(g)\notin \Im(N_{E/F})$. It is easy to see that we may choose $g\in H_0(F)$. By the proof of Proposition \ref{GU v.s. U}, we have $m(\pi_1)=m(\pi_2)$. By Lemma \ref{5.4}, we have $$m(\pi_1)+m(\pi_2)=m(\pi|_{G_1})=m_{geom}(\pi|_{G_1})=m_{geom}(\pi_1)+m_{geom}(\pi_2).$$
Hence in order to prove the multiplicity formula, it is enough to show that
\begin{equation}\label{5.6}
m_{geom}(\pi_1)=m_{geom}(\pi_2).
\end{equation}

There are two cases. When $\eta_{E/F}(-\varepsilon)=-1$, for any $T_1\in \CT_{ell}(H_{0,1})$, $Ad(g)T_1(F)$ is a maximal elliptic torus of $H_{0,1}(F)$ that is $H_{0,1}(F)$-conjugated to $T_1(F)$ (in fact, we may even choose $g$ properly so that $Ad(g)T_1(F)=T_1(F)$). This implies that
\begin{align}
&\int_{T_1(F)/Z_{G_1}(F)} D^{H_1}(t_1) c_{\pi_1}(t_1) \omega_1(t_1)^{-1}\ud_vt_1
\nonumber\\
=&\int_{T_1(F)/Z_{G_1}(F)} D^{H_1}(t_1) c_{\pi_2}(t_1) \omega_1(t_1)^{-1}\ud_vt_1. \label{5.10}
\end{align}
Meanwhile, since $\pi_2=\pi_1\circ Ad(g)$, we have
\begin{equation}\label{5.11}
c_{\pi_1}(1)=c_{\pi_2}(1).
\end{equation}
Then \eqref{5.6} is a direct consequence of \eqref{5.10} and \eqref{5.11}.

When $\eta_{E/F}(-\varepsilon)=1$, for $T_1\in \CT_{ell}(H_{0,1})$ that maps to a quadratic extension of $F$ that is not $E$, $Ad(g)T_1(F)$ is a maximal elliptic torus of $H_{0,1}(F)$ that is $H_{0,1}(F)$-conjugated to $T_1(F)$ (in fact, we may even choose $g$ properly so that $Ad(g)T_1(F)=T_1(F)$). This implies that
\begin{align}
&\int_{T_1(F)/Z_{G_1}(F)} D^{H_1}(t_1) c_{\pi_1}(t_1) \omega_1(t_1)^{-1}\ud_vt_1
\nonumber\\
=&\int_{T_1(F)/Z_{G_1}(F)} D^{H_1}(t_1) c_{\pi_2}(t_1) \omega_1(t_1)^{-1}\ud_vt_1. \label{5.7}
\end{align}
On the other hand, let $T_1,T_1'\in \CT_{ell}(H_{0,1})$ be the fiber of $E$ under the surjective map from $\CT_{ell}(H_{0,1})$ to the set of quadratic extensions of $F$. Then $Ad(g)T_1(F)$ (resp. $Ad(g)T_1'(F)$) is a maximal elliptic torus of $H_{0,1}(F)$ that is $H_{0,1}(F)$-conjugated to $T_1'(F)$ (resp. $T_1(F)$). This implies that $|W(H_{0,1},T_1)|=|W(H_{0,1},T_1')|$ and
\begin{align}
&\int_{T_1(F)/Z_{G_1}(F)} D^{H_1}(t_1) c_{\pi_1}(t_1) \omega_1(t_1)^{-1}\ud_vt_1+ \int_{T_1'(F)/Z_{G_1}(F)} D^{H_1}(t_1') c_{\pi_1}(t_1') \omega_1(t_1')^{-1}\ud_vt_1' \nonumber \\
=&\int_{T_1(F)/Z_{G_1}(F)} D^{H_1}(t_1) c_{\pi_2}(t_1) \omega_1(t_1)^{-1}\ud_vt_1+\int_{T_1'(F)/Z_{G_1}(F)} D^{H_1}(t_1') c_{\pi_2}(t_1') \omega_1(t_1')^{-1}\ud_vt_1'. \label{5.8}
\end{align}
Finally since $\pi_2=\pi_1\circ Ad(g)$, we have
\begin{equation}\label{5.9}
c_{\pi_1}(1)=c_{\pi_2}(1).
\end{equation}
Then \eqref{5.6} is a direct consequence of \eqref{5.7}, \eqref{5.8} and \eqref{5.9}. This finishes the proof of the proposition.
\end{proof}

Now we are ready to prove Theorem \ref{main U}. The argument is similar to the  unitary similitude group case. For $i=1,2$, fix $\varepsilon_i\in F^{\times}$ with $\eta_{E/F}(\varepsilon_i)=(-1)^{i-1}$ as before. Let $\Pi_{\phi}=\Pi_{\phi}(G_{1,\varepsilon_1})\cup \Pi_{\phi}(G_{1, \varepsilon_2})$ be a local tempered Vogan $L$-packet of $\RU_6(F)$. Our goal is to show that
\begin{equation}\label{5.5}
\sum_{i=1}^{2}\sum_{\pi_{1,\varepsilon_i}\in \Pi_{\phi}(G_{1,\varepsilon_i})} m(\pi_{1,\varepsilon_i}) =2.
\end{equation}

By the multiplicity formula in Proposition \ref{multiplicity formula unitary}, we have
\begin{align*}
&\sum_{i=1}^{2}\sum_{\pi_{1,\varepsilon_i}\in \Pi_{\phi}(G_{1,\varepsilon_i})} m(\pi_{1,\varepsilon_i})=\sum_{i=1}^{2}\sum_{\pi_{1,\varepsilon_i}\in \Pi_{\phi}(G_{1,\varepsilon_i}) } \left( 2c_{\pi_{1,\varepsilon_i}}(1)  \right.  \\
&\left. +\sum_{T_i\in \CT_{ell}(H_{0,1,\varepsilon_i})} \nu(T_i)\int_{T_i(F)/Z_{G_{1,\varepsilon_i}}(F)} D^{H_{1,\varepsilon_i}}(t_i)c_{\pi_{1,\varepsilon_i}}(t_i) \omega_{1,\varepsilon_i}^{-1}(t_i)\ud t_i \right)
\end{align*}
where $\nu(T_i)=|W(H_{0,1,\varepsilon_i},T_i)|^{-1} \vol(T_i(F)/Z_{G_{1,\varepsilon_i}}(F))^{-1}$. By the same argument as in the proof of Theorem \ref{main theorem for reduced models}, together with Theorem \ref{endoscopy U} \eqref{item:generic} and \eqref{item:endoscopy-3}, we can show that
\begin{align*}
&\sum_{i=1}^{2}\sum_{\pi_{1,\varepsilon_i}\in \Pi_{\phi}(G_{1,\varepsilon_i})} c_{\pi_{1,\varepsilon_i}}(1)=1,\\
&\sum_{i=1}^{2}\sum_{\pi_{1,\varepsilon_i}\in \Pi_{\phi}(G_{1,\varepsilon_i})} \sum_{T_i\in \CT_{ell}(H_{0,1,\varepsilon_i})} \nu(T_i)\int_{T_i(F)/Z_{G_{1,\varepsilon_i}}(F)} D^{H_{1,\varepsilon_i}}(t_i)c_{\pi_{1,\varepsilon_i}}(t_i) \omega_{1,\varepsilon_i}^{-1}(t_i)\ud t_i=0.
\end{align*}
This proves \eqref{5.5} and finishes the proof of Theorem \ref{main U}.

To summarize, we have reduced the proofs of the main theorems (i.e. Theorem \ref{main GU} and \ref{main U}) to the proof of the trace formula in Theorem \ref{trace formula}.

\section{The proof of the spectral side of the trace formula}\label{sec:proof-spectral}
In this section, we will prove the trace formula. Since the geometric side of the trace formula has already been proved in Section \ref{sec:proof-geometric}, we only need to consider the spectral side. As in the previous sections, we will use $(G,H)$ to denote $(G_{\varepsilon},H_{\varepsilon})$ and use $\omega\otimes \xi$ to denote $\omega_{\varepsilon}\otimes \xi_{\varepsilon}$.
\subsection{Explicit intertwining operator $\CL_{\pi}$} \label{sec:intertwining}
By Lemma \ref{major 2}, for all $f\in \CC(\zg,\eta^{-1})$, the integral
$$\int_{\zh}f(h)\omega\otimes \xi(h)\ud h$$
is absolutely convergent and defines a continuous linear form on the space $\CC(\zg,\eta^{-1})$.

\begin{prop}\label{major 8}
\begin{enumerate}
\item The linear form
$$f\in \CC(\zg,\eta^{-1}) \rightarrow \int_{\zh}f(h)\omega\otimes \xi(h)\ud h$$
can be extended continuously to $\CC^w(\zg,\eta^{-1})$. We will use
$$\CP_{H,\omega\otimes\xi}:\;f\in \CC^{w}(\zg,\eta^{-1})\rightarrow \int_{\zh}^{\ast}f(h)\omega\otimes \xi(h) \ud h$$
to denote this continuous linear form.
\item For all $f\in \CC^{w}(\zg,\eta^{-1})$, and $h_0,h_1\in H(F)$, we have
$$\CP_{H,\omega\otimes\xi}(L(h_0)R(h_1) f)=\omega\otimes \xi(h_0h_{1}^{-1}) \CP_{H,\omega\otimes\xi}(f)$$
where $R$ (resp. $L$) is the right (resp. left) translation.
\end{enumerate}
\end{prop}

\begin{proof}
This proof is very similar to the GR case (Proposition 5.1 and Lemma 5.2 of \cite{Wan16}), and we will skip it here.
\end{proof}

Let $\pi$ be a tempered representation of $G(F)$ with central character $\eta$. For all $T\in \End(\pi)^{\infty}$, define
$$\CL_{\pi}(T)=\CP_{H,\xi}(\tr(\pi(g^{-1})T))=\int_{\zh}^{\ast} \tr(\pi(h^{-1})T)\omega\otimes \xi(h)\ud h.$$
By Proposition \ref{major 8}, the map $\CL_{\pi}: \End(\pi)^{\infty}\rightarrow \BC$ is a continuous linear form in $\End(\pi)^{-\infty}$. Here $\End(\pi)^{-\infty}$ is the topological dual of $\End(\pi)^{\infty}$ endowed with the strong topology. By Proposition \ref{major 8}, for any $h,h'\in H(F)$, we have
\begin{equation}\label{L 1}
\CL_{\pi}(\pi(h)T\pi(h'))=\omega\otimes \xi(hh')\CL_{\pi}(T).
\end{equation}
For $e,e'\in \pi$, define $T_{e,e'}\in \End(\pi)^{\infty}$ to be the map
$e_0\in \pi\mapsto (e_0,e')e.$
Set $\CL_{\pi}(e,e')=\CL_{\pi}(T_{e,e'})$. Then
$$\CL_{\pi}(e,e')=\int_{\zh}^{\ast} (e,\pi(h)e')\omega\otimes \xi(h) \ud h.$$
If we fix $e'$, by \eqref{L 1}, the map $e\in \pi\rightarrow \CL_{\pi}(e,e')$ belongs to $Hom_H(\pi,\omega\otimes \xi)$. Since $Span \{T_{e,e'}\colon e,e'\in \pi\}$ is dense in $\End(\pi)^{\infty}$, we have
\begin{equation}\label{L(pi) implies m(pi)}
\CL_{\pi}\neq 0\Rightarrow m(\pi)\neq 0.
\end{equation}
Later in this subsection, we will show that the other direction also holds. Before that, we discuss some basic properties of $\CL_{\pi}$.

\begin{lem}\label{lemma 1}
With the notation above, the followings hold.
\begin{enumerate}
\item The map
$\pi \in \Pi_{temp}(G,\eta)\rightarrow \CL_{\pi}\in \End(\pi)^{-\infty}$
is smooth.
\item For $f\in \CC(\zg,\eta^{-1})$, we have
$$\int_{\zh}f(h) \omega\otimes \xi(h)\ud h=\int_{\Pi_{temp}(G,\eta)} \CL_{\pi}(\pi(f))\mu(\pi) \ud\pi$$
with both integrals being absolutely convergent.
\item For $f\in \CC_{ind}(\zg,\eta^{-1})$ and $f'\in \CC(\zg,\eta)$, we have
\begin{align*}
&\int_{\Pi_{temp}(G,\eta)} \CL_{\pi}(\pi(f)) \overline{\CL_{\pi}(\pi(\bar{f'}))} \mu(\pi) \ud\pi\\
=&\int_{\zh}\int_{\zh}\int_{\zg} f(hgh')f'(g)\ud g \omega\otimes \xi(h') \ud h'\omega\otimes \xi(h)\ud h
\end{align*}
where the left hand side is absolutely convergent and the right hand side is convergent in that order but is not necessarily absolutely convergent.
\end{enumerate}
\end{lem}

\begin{proof}
The proof is similar to the GGP case (Lemma 8.2.1 of \cite{B15}) and the GR case (Lemma 6.2.2 of \cite{Wan17}), so we will skip it here. The only thing we want to point out is that the proof of (3) in both previous cases uses the Gelfand pair condition (i.e. $m(\pi)\leq 1$ for all irreducible tempered representations $\pi$). But we don't have this condition for the current model, all we have is that $m(\pi)\leq 1$ for all $\pi\in \Pi_{temp}(G,\eta)\smallsetminus\Pi_2(G,\eta)$ (see Corollary \ref{Gelfand pair cor}). This is why we require $f\in \CC_{ind}(\zg,\eta^{-1})$ in the statement of (3) (note that in the previous two cases, the equality in (3) holds for all $f\in \CC(\zg,\eta^{-1})$).
\end{proof}

We then study the behavior of $\CL_{\pi}$ under the parabolic induction. Let $\bar{Q}=M_Q\bar{U}_{Q}$ be a good parabolic subgroup, and $\tau\in \Pi_2(M_Q)$ be a discrete series whose central character equals $\eta$ on $Z_G(F)$. Set $\pi=I_{\bar{Q}}^{G}(\tau)$. Then $\pi$ is a tempered representation of $G$ with central character $\eta$. Let $H_{\bar{Q}}=H\cap \bar{Q}$. For $T\in \End(\tau)^{\infty}$, define
$$\CL_{\tau}(T)=\int_{Z_H(F)\back H_{\bar{Q}}(F)} \tr(\tau(h_{\bar{Q}}^{-1})T)\delta_{H_{\bar{Q}}}(h_{\bar{Q}})^{1/2} \omega\otimes \xi(h_{\bar{Q}}) \ud h_{\bar{Q}}.$$
The integral above is absolutely convergent by Proposition \ref{major 5} \eqref{item:major-5-2} together with the assumption that $\tau$ is a discrete series.

\begin{prop}\label{parabolic induction}
With the notation above, we have
$$\CL_{\pi}\neq 0 \iff \CL_{\tau}\neq 0.$$
\end{prop}

\begin{proof}
The proof is very similar to the GR model case (Proposition 5.6 of \cite{Wan16}), so we will skip it here.
\end{proof}

The following proposition tells us the relation between $\CL_{\pi}$ and $m(\pi)$.
\begin{prop}\label{m(pi) L(pi)}
Let $\pi$ be an irreducible tempered representation of $G(F)$ with central character $\eta$. Then
$\CL_{\pi}\neq 0 \iff m(\pi)\neq 0.$
\end{prop}

\begin{proof}
By \eqref{L(pi) implies m(pi)}, we only need to show that $m(\pi)\neq 0\Rightarrow \CL_{\pi}\neq 0$. From now on, we assume that $m(\pi)\neq 0$. Fix $0\neq l\in Hom_{H(F)}(\pi,\omega\otimes\xi)$. For $T\in C_{c}^{\infty}(\Pi_{temp}(G,\eta))$, by the same argument as in the GR model case (Section 5.5 of \cite{Wan16}), we have
\begin{equation}\label{thm 1.6}
l(T_{\pi}e)=\int_{H(F)\back G(F)} l(\pi(x)e)\int_{\Pi_{temp}(G,\eta)} \CL_{\Pi}(T_{\Pi} \Pi(x^{-1})) \mu(\Pi) \ud\Pi \ud x
\end{equation}
for all $e\in \pi$. Choose a good parabolic subgroup $\bar{Q}=M_Q\bar{U}_Q$ of $G$ and $\tau\in \Pi_2(M_Q)$ such that $\pi$ is a direct summand of $\pi'=I_{\bar{Q}}^{G}(\tau)$. Let
\begin{equation}\label{connected component}
\CO=\{ I_{\bar{Q}}^{G}(\tau_{\lambda})\colon \lambda\in i\Fa_{M_Q,0}^{\ast}\}\subset \Pi_{temp}(G,\eta)
\end{equation}
be the connected component containing $\pi'$. Choose $e_0\in \pi$ such that $l(e_0)\neq 0$, and let $T_0\in \End(\pi)^{\infty}$ with $T_0(e_0)=e_0$. We can easily find an element $T^0\in C_{c}^{\infty}(\Pi_{temp}(G,\eta))$ such that
$$T^{0}_{\pi}=T_0,\; Supp(T^0)\subset \CO.$$
By applying \eqref{thm 1.6} to the case $e=e_0$ and $T=T^0$, we know there exists $\lambda \in i\Fa_{M_Q,0}^{\ast}$ such that $\CL_{\pi_{\lambda}'}\neq 0$ where $\pi_{\lambda}'=I_{\bar{Q}}^{G}(\tau_{\lambda})$. By Proposition \ref{parabolic induction}, this implies that $\CL_{\tau_{\lambda}} \neq 0$. We need a lemma:

\begin{lem}\label{reduce model}
For all $\lambda\in i\Fa_{M_Q,0}^{\ast}$, we have
$$\CL_{\tau}\neq 0\iff\CL_{\tau_{\lambda}} \neq 0.$$
\end{lem}

\begin{proof}
When the reduced model $(G_{\bar{Q}},H_{\bar{Q}})$ is of Type I (defined in Page \pageref{pg:type-I}), it is easy to see from the definition that the nonvanishing property of $\CL_{\tau}$ is invariant under the unramified twist.

When the reduced model is of Type II, it is not clear from the definition that the unramified twist will preserve the nonvanishing property. Instead, we claim that $\CL_{\tau}$ is always nonzero in this case. In fact, since $\tau$ is a discrete series, by the same argument as in the GR model case (i.e. Remark 5.12 of \cite{Wan16}), we have $m(\tau)\neq 0\Rightarrow \CL_{\tau}\neq 0$. Hence it is enough to show that the multiplicity is always nonzero for Type II reduced models. But this just follows from Theorem \ref{main theorem for reduced models}.
\end{proof}
Back to the proof of the proposition. The lemma above implies that $\CL_{\tau}\neq 0$. Together with Proposition \ref{parabolic induction}, we have $\CL_{\pi'}\neq 0$. Since $\pi$ is a direct summation of $\pi'$, we can write $\pi'$ as $\pi\oplus \pi_0$. If $\CL_{\pi_0}\neq 0$, then $m(\pi_0)\neq 0$ which implies that $m(\pi')=m(\pi)+m(\pi_0)\geq 2$. Meanwhile, by Corollary \ref{Gelfand pair cor}, we know that $m(\pi')\leq 1$. So we get a contradiction. Hence we have $\CL_{\pi_0}=0$. But since $\CL_{\pi'}\neq 0$, we have $\CL_{\pi}\neq 0$. This finishes the proof of the proposition.
\end{proof}

Now we are ready to prove Theorem \ref{multiplcity parabolic induction} by assuming Proposition \ref{Gelfand pair} holds. We first recall the statement of the theorem. Let $\bar{Q}=M_Q\bar{U}_{Q}$ be a good parabolic subgroup of $G$, and let $\tau$ be an admissible tempered representation $M_{Q}(F)$ whose central character equals $\eta$ on $Z_G(F)$. Set $\pi=I_{\bar{Q}}^{G}(\tau)$. Our goal is to show that $m(\pi)=m(\tau)$. Without loss of generality, we can assume that $\tau$ is irreducible. By Theorem \ref{main theorem for reduced models} and Proposition \ref{Gelfand pair}, it is enough to show that
\begin{equation}\label{multiplcity parabolic induction 3}
m(\tau)\neq 0 \Rightarrow m(\pi)\neq 0.
\end{equation}
By applying the same argument in this section to the reduced models, we can define the operator $\CL_{\tau}$ via a regularized integral and we can prove analogy results of Proposition \ref{parabolic induction} and Theorem \ref{m(pi) L(pi)} for the reduced models. As a result, we have
$$m(\tau)\neq 0\iff \CL_{\tau}\neq 0\iff \CL_{\pi}\neq 0\iff m(\pi)\neq 0.$$
This proves \eqref{multiplcity parabolic induction 3} and finishes the proof of Theorem \ref{multiplcity parabolic induction}.

To end this subsection, we prove a proposition that will be used later in the proof of the spectral side of the trace formula.
\begin{prop}\label{parabolic induction 5}
Let $\CK\subset \Pi_{temp}(G,\eta)$ be a compact subset. Then there exists an element $T\in \CC(\Pi_{temp}(G,\eta))$ such that
$\CL_{\pi}(T_{\pi})=m(\pi)$ for all $\pi\in \CK$.
\end{prop}

\begin{proof}
It is enough to show that for all $\pi' \in \Pi_{temp}(G,\eta)$, there exists $T\in \CC(\Pi_{temp}(G,\eta))$ such that
$\CL_{\pi}(T_{\pi})=m(\pi)$
for all $\pi$ in some neighborhood of $\pi'$ in $\Pi_{temp}(G,\eta)$. If $\pi'$ is a discrete series, by Theorem \ref{m(pi) L(pi)}, we can find $T'\in \End(\pi')^{\infty}$ such that $\CL_{\pi'}(T')=m(\pi')$. Then we just need to take any $T\in \CC(\Pi_{temp}(G,\eta))$ with $T_{\pi'}=T'$.

If $\pi'$ is not a discrete series, we can find a good parabolic subgroup $\bar{Q}=M_Q\bar{U}_{Q}$ of $G$ and $\tau\in \Pi_2(M_Q)$ such that $\pi'=I_{\bar{Q}}^{G}(\tau)$. Let $\CO=\{ I_{\bar{Q}}^{G}(\tau_{\lambda})\colon \lambda\in i\Fa_{M_Q,0}^{\ast}\}$ be the connected component containing $\pi'$. We first show that
\begin{description}
\item[(1)] The multiplicity is constant on $\CO$ (i.e. $m(\pi)=m(\pi')$ for all $\pi\in \CO$).
\end{description}

By Theorem \ref{multiplcity parabolic induction}, it is enough to show that the multiplicity $m(\tau)$ of the reduced model $(G_{\bar{Q}},H_{\bar{Q}})$ is invariant under unramified twist. When the reduced model is of Type I, this just follows from the definition. If the reduced model is of Type II, by Theorem \ref{main theorem for reduced models}, the multiplicity is always equal to $1$ and hence invariant under the unramified twist. This proves (1).

Now if $m(\pi')=0$, then $m(\pi)=0$ for all $\pi\in \CO$ and we can just take $T=0$. If $m(\pi')\neq 0$, by the discussion above together with Corollary \ref{Gelfand pair cor}, we know that $m(\pi)=1$ for all $\pi\in \CO$. By Theorem \ref{m(pi) L(pi)}, we can find $T'\in \End(\pi')^{\infty}$ such that $\CL_{\pi'}(T')\neq 0$. Let $T^0\in \CC(\Pi_{temp}(G,\eta))$ be an element with $T_{\pi'}^{0}=T'$. By Lemma \ref{lemma 1}(1), the function $\pi\rightarrow \CL_{\pi}(T^{0}_{\pi})$ is a smooth function. The value at $\pi'$ is just $\CL_{\pi'}(T')\neq 0$. Hence we can find a smooth compactly supported function $\varphi$ on $\Pi_{temp}(G,\eta)$ such that
$\varphi(\pi)\CL_{\pi}(T^{0}_{\pi})=1$
for all $\pi$ belonging to a small neighborhood of $\pi'$. Then we just need to take $T=\varphi T^0$ and this finishes the proof of the Proposition.
\end{proof}

\subsection{The proof of the spectral side of the trace formula}
In this subsection, we will prove the spectral side of the trace formula. Since we only know $m(\pi)\leq 1$ for $\pi\in \Pi_{temp}(G)\smallsetminus\Pi_{2}(G)$, we need to divide the proof into two steps.

\textbf{Step 1:} We first prove the spectral side of the trace formula for $f\in {}^{\circ}\CC(\zg,\eta^{-1})$. For such $f$, the spectral side only contains discrete series. By a similar argument as in the Galois model case (Section 3 of \cite{B17}), we can prove the trace formula without using the multiplicity one assumption.

\textbf{Step 2:}  We then prove the trace formula for $f\in \CC_{ind,scusp}(\zg,\eta^{-1})$. For this kind of test functions, the spectral side will only contain tempered representations which are not discrete series. But for these representations, the multiplicity one assumption holds by Corollary \ref{Gelfand pair cor}. Hence we can prove the trace formula by the same argument as in the GGP case and the GR case.

We start with Step 1. For $\pi\in \Pi_2(G,\eta)$, let
$$\mathcal{B}_\pi:\pi\times \pi^\vee\to \BC$$
be the bilinear form defined by
$$\mathcal{B}_\pi(v,v^\vee):=\int_{\zh} \langle v,\pi(h)v^\vee\rangle (\omega\otimes \xi)(h)\ud h,\;(v,v^\vee)\in \pi\times \pi^\vee.$$
Note that the integral above is absolutely convergent by Lemma \ref{major 2}. It is easy to see from the definition that $\mathcal{B}_\pi$ descents to a bilinear pairing
$$\mathcal{B}_\pi: \pi_{\omega\otimes\xi} \times \pi^\vee_{(\omega\otimes \xi)^{-1}}\to \BC$$
where $\pi_{\omega\otimes\xi}$ (resp. $\pi^\vee_{(\omega\otimes \xi)^{-1}}$) is the $(H,\omega\otimes \xi)$-coinvariant spaces  (resp. $(H,(\omega\otimes \xi)^{-1})$-coinvariant)  of $\pi$ (resp. $\pi^\vee$).

\begin{prop}\label{prop intertwinings}
$\mathcal{B}_\pi$ induces a perfect pairing between $\pi_{\omega\otimes\xi}$ and $\pi^\vee_{(\omega\otimes \xi)^{-1}}$.
\end{prop}

\begin{proof}
This proposition follows from exactly the same argument as Proposition 3.2.1 of \cite{B17}  after we establish the next lemma.

\begin{lem}\label{lemma intertwinings}
For all $l\in Hom_H(\pi,\omega\otimes \xi)$, $v\in \pi$ and $f\in \mathcal{C}(\zg,\eta^{-1})$, the integrals
$$\int_{H(F)\backslash G(F)} \lvert l(\pi(x)v)\rvert^2 \ud x
\text{ and }
\int_{\zg} f(g) l(\pi(g)v) \ud g$$
are absolutely convergent. Moreover, we have
\begin{equation}\label{lemma intertwinings 2}
l(\pi(f)v)=\int_{\zg} f(g) l(\pi(g)v) \ud g.
\end{equation}
\end{lem}

\begin{proof}
By a similar argument as in the GGP case (Lemma 8.3.1 of \cite{B15}) or the GR case (Lemma 6.2.3 of \cite{Wan17}), we have
\begin{equation}\label{lemma intertwinings 1}
| l(\pi(x)e)| \ll \hc(x)\nor(x)^{-d}
\end{equation}
for all $e\in \pi$ and $x\in H(F)\back G(F)$. Note that both loc. cit. considered all tempered representations, and hence the right hand side in both loc. cit. is $\hc(x)\nor(x)^{d}$. Here since we only consider discrete series, we can have better bound on the right hand side (i.e. $\hc(x)\nor(x)^{-d}$). Combining \eqref{lemma intertwinings 1} with Proposition \ref{major 4}, we know that both integrals in the lemma are absolutely convergent. Finally, \eqref{lemma intertwinings 2} follows from the standard argument as in the GGP case (Section 8.5 of \cite{B15}) and the GR case (Section 6.5 of \cite{Wan17}).
\end{proof}
\end{proof}

Now we are ready to prove the spectral side of the trace formula for $f\in {}^{\circ}\CC(\zg,\eta^{-1})$. We fix such a test function $f$. Without loss of generality, we may assume that there exist $\pi\in \Pi_2(G,\chi)$, $v\in \pi$ and $v^{\vee}\in \pi^{\vee}$ such that $f(g)=\langle v,\pi(g)v^{\vee} \rangle$. Then the spectral side becomes
$$I_{spec}(f)=\tr(\pi(f)) m(\bar{\pi})=\tr(\pi(f))m(\pi).$$
For $x,y\in G(F)$, define
$$K_f(x,y)=\int_{\zh} f(x^{-1}hy)\omega\otimes \xi(h)\ud h.$$
Then we have
$$K_f(x,y)=\CB_{\pi}(\pi(x)v,\pi^{\vee}(y)v') \text{ and } K_f(x,x)=I(f,x).$$
Let $N=m(\pi)$ and let $v_1,\cdots,v_N$ be vectors in $\pi$ whose images in $\pi_{\omega\otimes \xi}$ form a basis. We then let $v_{1}^{\vee},\cdots,v_{N}^{\vee}$ be vectors in $\pi^{\vee}$ whose images in $\pi_{(\omega\otimes \xi)^{-1}}^{\vee}$ form the dual basis under the pairing $\CB_{\pi}$. As in the Galois model case (Section 3 of \cite{B17}), we have
\begin{eqnarray*}
I(f)&=&\int_{H(F)\back G(F)}I(f,x)\ud x=\int_{H(F)\back G(F)} \CB_{\pi}(\pi(x)v,\pi^{\vee}(x)v')\ud x\\
&=&\sum_{i=1}^{N} \int_{H(F)\back G(F)} \CB_{\pi}(\pi(x)v,v_{i}^{\vee}) \CB_{\pi}(v_i, \pi^{\vee}(x)v') \ud x\\
&=&\sum_{i=1}^{N} \int_{\zg} \langle\pi(g)v, v_{i}^{\vee}\rangle \CB_{\pi}(v_i,\pi^{\vee}(g)v^{\vee})\ud g\\
&=&\sum_{i=1}^{N} \frac{\langle v,v^{\vee}\rangle}{d(\pi)} \CB_{\pi}(v_i,v_{i}^{\vee})=\tr(\pi(f)) \cdot N=\tr(\pi(f))m(\pi)
\end{eqnarray*}
where $d(\pi)$ is the formal degree of $\pi$. This proves the spectral side of the trace formula for all $f\in {}^{\circ}\CC(\zg,\eta^{-1})$.

\vspace{2em}
Now it remains to consider the case when $f\in \CC_{ind,scusp}(\zg,\eta^{-1})$. We fix such a function $f$. Then the spectral side does not contain discrete series. By Corollary \ref{Gelfand pair cor}, we have the Gelfand pair condition for all the representations appear in the spectral side. As a result, we can prove the trace formula by a similar argument as in the GGP case and the GR case. To be specific, for $f'\in \CC(\zg,\eta)$, define
\begin{align*}
&K_{f,f'}(g_1,g_2)= \int_{\zg} f(g_{1}^{-1} gg_2)f'(g) \ud g, \;g_1,g_2\in G(F),\\
&K_{f,f'}^{H}(x,y)=\int_{\zh}\int_{\zh}K_{f,f'}(h_{1}^{-1}x,h_2y) \omega\otimes \xi(h_1h_2) \ud h,\; x,y\in G(F),\\
&J_{aux}(f,f')=\int_{H(F)\back G(F)} K_{f,f'}^{H}(x,x)\ud x.
\end{align*}

\begin{prop}\label{spectral 1}
All three integrals above are absolutely convergent. Moreover, we have
\begin{equation}\label{spectral 1.1}
K_{f,f'}^{H}(x,y)=\int_{\Pi_{temp}(G,\eta)} \CL_{\pi}(\pi(x)\pi(f)\pi(y^{-1}))\overline{\CL_{\pi}(\pi(\overline{f'}))} \mu(\pi)\ud\pi,
\end{equation}
\begin{equation}\label{spectral 1.2}
J_{aux}(f,f')=\int_{\CX(G,\eta)} D(\pi)\theta_f(\pi) \overline{\CL_{\pi}(\pi(\overline{f'}))} \ud\pi.
\end{equation}
\end{prop}

\begin{proof}
The proof is very similar to the GGP case (Proposition 9.2.1 and 9.2.2 of \cite{B15}) and the GR case (Proposition 8.2.2 and 8.2.3 of \cite{Wan17}), we will skip it here. The only thing we want to point out is that the proof of \eqref{spectral 1.1} uses Lemma \ref{lemma 1}(3) whose proof uses the Gelfand pair condition. This is why we need to require $f\in \CC_{ind,scusp}(\zg,\eta^{-1})$.
\end{proof}

Now we are ready to prove the trace formula. By Lemma \ref{lemma 1}(2), together with the fact that $f\in \CC_{ind,scusp}(\zg,\eta^{-1})$, we have
\begin{equation}\label{8.10}
I(f,x)=\int_{\Pi_{temp}(G,\eta)\smallsetminus\Pi_2(G,\eta)} \CL_{\pi}(\pi(x)\pi(f)\pi(x)^{-1})\mu(\pi) \ud\pi.
\end{equation}
By Proposition \ref{parabolic induction 5}, there exists a function $f'\in \CC(\zg,\eta)$ such that
$$\CL_{\pi}(\pi(\overline{f'}))=m(\pi)$$
for all $\pi\in \Pi_{temp}(G,\eta)$ with $\pi(f)\neq 0$ (note that $\pi(f)\neq 0$ will imply $\pi\in \Pi_{temp}(G,\eta)\smallsetminus\Pi_2(G,\eta)$). We fix such a function $f'$. By Theorem \ref{m(pi) L(pi)} and Corollary \ref{Gelfand pair cor}, for all $\pi\in \Pi_{temp}(G,\eta)\smallsetminus\Pi_2(G,\eta)$), $\CL_{\pi}\neq 0$ if and only if $m(\pi)=1$. Then \eqref{8.10} becomes
$$I(f,x)=\int_{\pi\in \Pi_{temp}(G,\eta)\smallsetminus\Pi_2(G,\eta))} \CL_{\pi}(\pi(x)\pi(f)\pi(x)^{-1})\overline{\CL_{\pi}(\pi(\overline{f'}))} \mu(\pi) \ud\pi.$$
Combining with Proposition \ref{spectral 1}, we have $I(f,x)=K_{f,f'}^{H}(x,x)$. Therefore $I(f)=J_{aux}(f,f')$. Applying Proposition \ref{spectral 1} again, together with the fact that $\overline{\CL_{\pi}(\pi(\overline{f'}))}=m(\pi)=m(\bar{\pi})$, we have
$$I(f)=J_{aux}(f,f')=\int_{\CX(G,\eta)} D(\pi)\theta_f(\pi)m(\bar{\pi}) \ud\pi=I_{spec}(f).$$
This finishes the proof of the spectral side of the trace formula.

\appendix

\section{The proof of Proposition \ref{Gelfand pair}} \label{sec:appendix}
\subsection{Some reduction}
We first recall the notation and the statement of the proposition. Let $\bar{Q}=M_Q\bar{U}_{Q}$ be a good parabolic subgroup of $G$, and $\tau$ be an admissible tempered representation of $M_Q(F)$ whose central character equals $\eta$ on $Z_G(F)$. Set $\pi=I_{\bar{Q}}^{G}(\tau)$. Our goal is to show that
\begin{equation}\label{A.1}
m(\pi)\leq m(\tau).
\end{equation}
Here $m(\pi)$ is the multiplicity for the model $(G,H)$ and $m(\tau)$ is the multiplicity for the reduced model $(G_{\bar{Q}},H_{\bar{Q}})$. We will apply the orbit method to prove \eqref{A.1}. To simplify our notation, we assume that the characters $\eta$ and $\omega$ are trivial. The argument for the general case (i.e. when the characters are nontrivial) will be exactly the same as this case.

First we consider the analogue of \eqref{A.1} for reduced models. Let $M_Q\bar{U}_Q=\bar{Q}\subset \bar{Q}'=M_Q' \bar{U}_Q'$ be two good parabolic subgroups, and let $(G_{\bar{Q}'},H_{\bar{Q}'}), \; (G_{\bar{Q}},H_{\bar{Q}})$ be the associated reduced models. Given an admissible tempered representation $\tau$ of $M_{\bar{Q}}(F)$, set $\tau'=I_{\bar{Q}\cap M_Q'}^{M_Q'}(\tau)$ which is an admissible tempered representation of $M_{\bar{Q}}'(F)$. Let $m(\tau)$ and $m(\tau')$ be the multiplicity of the reduced models $(G_{\bar{Q}},H_{\bar{Q}})$ and $(G_{\bar{Q}'},H_{\bar{Q}'})$ respectively. Then the analogue of \eqref{A.1} for reduced models is just
\begin{equation}\label{A reduced model}
m(\tau')\leq m(\tau).
\end{equation}
The proof of \eqref{A reduced model} follows from the same, but easier arguments as the proof of \eqref{A.1}. Hence by induction, we will assume that \eqref{A reduced model} holds for all good parabolic subgroups $\bar{Q}\subset \bar{Q}'$.

Then we reduce the proof of \eqref{A.1} to the case when $\bar{Q}$ is a maximal parabolic subgroup. In general, if $\bar{Q}$ is not a maximal parabolic subgroup, we can find a good maximal parabolic subgroup $\bar{Q}'=M_Q'\bar{U}_Q'$ with $\bar{Q}\subset \bar{Q}'$. Let $\tau'=I_{\bar{Q}\cap M_Q'}^{M_Q'}(\tau)$ be an admissible tempered representation of $M_Q'(F)$ and let $m(\tau')$ be the multiplicity for the reduced model $(G_{\bar{Q}'},H_{\bar{Q}'})$. Then we have $\pi=I_{\bar{Q}'}^{G}(\tau')$. Once we have proved \eqref{A.1} for all maximal parabolic subgroups, we have $m(\pi)\leq m(\tau')$. By \eqref{A reduced model}, we also have $m(\tau')\leq m(\tau)$. This proves \eqref{A.1} for the parabolic subgroup $\bar{Q}$. So from now on, we can assume that $\bar{Q}$ is a maximal good parabolic subgroup. Also by Remark \ref{irreducible}, we may also assume that $\tau$ is irreducible.

Since $G/H$ is a spherical variety, the double coset $\bar{Q}(F)\back G(F)/H(F)$ only contains finitely many elements, and  we denote it by $\{\bar{Q}\gamma_iH \colon 1\leq i\leq k\}$. By the geometric lemma of Bernstein-Zelevinsky in \cite{BZ77}, we may reorder $\gamma_i$ such that
$$Y_i=\cup_{j=1}^{i} \bar{Q}(F)\gamma_j H(F)$$
is an open subset of $G(F)$ for all $1\leq i\leq k$. Since $\bar{Q}$ is a good parabolic subgroup, we may assume that $\gamma_1=1$.

With the filtration above, for $1\leq i\leq k$, define
$$V_i=\{f\in I_{\bar{Q}}^{G}(\tau)\colon supp(f)\subset Y_i\}.$$
Then we have $V_1\subset V_2\subset \cdots \subset V_k=\pi$ and $V_i$ is $H(F)$-invariant for all $i$. In particular, this implies that
\begin{equation}\label{A.3}
m(\pi)=\dim(\Hom_{H(F)}(\pi,\xi)) \leq \sum_{i=1}^{k}\dim(\Hom_{H(F)}(V_i/V_{i-1},\xi)).
\end{equation}
Here $V_0=\{0\}$. Moreover, for any $1\leq i\leq k$, it is easy to see that the map
$$f\in V_i\mapsto \phi_f(h):=f(y_i h)$$
is an isomorphism between $V_i/V_{i-1}$ and $ind_{H_i}^{H} (\delta_{P}^{1/2} \tau^{\gamma_i}|_{H_i})$ ($ind_{H_i}^{H}$ is the compact induction). Here $H_i=H(F)\cap \gamma_{i}^{-1}\bar{Q}(F)\gamma_i=\gamma_{i}^{-1} Q_i(F) \gamma_i$ with $Q_i(F)=\bar{Q}(F)\cap \gamma_{i}H(F)\gamma_{i}^{-1}$. By reciprocity law, we have
\begin{equation}\label{reciprocity}
\Hom_{H(F)}(V_i/V_{i-1},\xi)\simeq \Hom_{Q_i(F)} (\tau, (\delta_{Q_i}\delta_{\bar{Q}}^{-1/2})\otimes {}^{\gamma_i}\xi)
\end{equation}
where the character ${}^{\gamma_i}\xi$ is defined by ${}^{\gamma_i}\xi(q):=\xi(\gamma_{i}^{-1}q \gamma_i)$ for $q\in Q_i(F)$. Here we view $\tau$ as a representation of $\bar{Q}(F)$ by making it trivial on $\bar{U}_Q(F)$.

\begin{prop}\label{A.4}
With the notation above, we have
$$m(\tau)=\dim(\Hom_{H(F)}(V_1,\xi)).$$
\end{prop}

\begin{proof}
Since $\gamma_1=1$, we have $Q_1=H\cap \bar{Q}=H_{\bar{Q}}$. By \eqref{reciprocity}, we have
$$\Hom_{H(F)}(V_1,\xi)\simeq \Hom_{H_{\bar{Q}}(F)} (\tau, (\delta_{H_{\bar{Q}}}\delta_{\bar{Q}}^{-1/2})\otimes \xi|_{H_{\bar{Q}}} ).$$
By Proposition \ref{major 5}, $\delta_{\bar{Q}}|_{H_{\bar{Q}}}=\delta_{H_{\bar{Q}}}$. Hence we have
$$\dim(\Hom_{H(F)}(V_1,\xi))=\dim(\Hom_{H_{\bar{Q}}(F)} (\tau, \delta_{H_{\bar{Q}}}^{1/2}\otimes \xi|_{H_{\bar{Q}}} )) =m(\tau)$$
where the last equality is the definition of $m(\tau)$.
\end{proof}

The following proposition will be proved in the next subsection.
\begin{prop}\label{A.2}
With the notation above, for all $2\leq i\leq k$, we have
$$\Hom_{H(F)}(V_i/V_{i-1},\omega)=\Hom_{Q_i(F)} (\tau, (\delta_{Q_i}\delta_{\bar{Q}}^{-1/2})\otimes {}^{\gamma_i} \xi)=\{0\}.$$
In other words, all the non-open orbits are not distinguished.
\end{prop}

Combine Proposition \ref{A.4} and \ref{A.2} with the inequality \eqref{A.3}, we have
$$m(\pi)\leq \dim(\Hom_{H(F)}(V_1,\xi))=m(\tau).$$
This proves \eqref{A.1}. Hence it remains to prove Proposition \ref{A.2}.

\subsection{The proof of Proposition \ref{A.2}}
In this subsection, we are going to prove Proposition \eqref{A.2}. In other words, we need to show that the representation $\tau$ of $\bar{Q}(F)$ is not $(Q_i, (\delta_{Q_i}\delta_{\bar{Q}}^{-1/2})\otimes {}^{\gamma_i}\xi)$-distinguished for $2\leq i\leq k$ (i.e. all the non-open orbits of $\bar{Q}(F)\back G(F)/H(F)$ are not distinguished). We will only study the quasi-split case here, and the non quasi-split case follows from a similar but easier argument (this is because there are less orbits in $\bar{Q}(F)\back G(F)/H(F)$ for the non quasi-split case). To simplify the computation, we will use the matrix $w_{2n}$ instead of $J_{2n,\varepsilon}$ to define the even  unitary similitude group, and we set $\GU_{n,n}=\GU(w_{2n})$ (in particular, $G(F)=\GU_{3,3}(F)$). Since $\bar{Q}$ is a maximal parabolic subgroup, its Levi part $M_{Q}(F)$ is isomorphic to $\GL_2(E)\times \GU_{1,1}(F)$, $\GL_1(E)\times \GU_{2,2}(F)$ or $\GL_3(E)\times \GL_1(F)$.

We first consider the case when $M_{Q}(F)\simeq \GL_2(E)\times \GU_{1,1}(F)$. In this case, we may just take  $\bar{Q}=\bar{P}$ and $M_Q=M$ where $\bar{P}$ is the parabolic subgroup opposite to $P$ and $P=MU$ is the standard parabolic subgroup of $G$ defined in the introduction. We need to compute the double coset $\bar{P}(F)\back G(F)/H(F)$. By the Bruhat decomposition, the double coset $\bar{P}(F)\back G(F)/P(F)$ contains $5$ elements $\{\bar{P}(F)v_iP(F)|\; 1\leq i\leq 5\}$ with
$$v_1=I_6,\; v_2=w_{(653421)},\;v_3=w_{(623451)},\;v_4=w_{(351624)},v_5=w_{(321654)}.$$
Here we are using partitions to denote the Weyl elements. To be specific, $w_{(s_1,\cdots,s_n)}$ is the $n$-by-$n$ matrix with entries $1$ in the $(s_k,k)$ positions ($1\leq k\leq n$) and $0$ elsewhere. Then we need to break the orbit $\bar{P}(F)v_iP(F)$ into orbits in $\bar{P}(F)\back G(F)/H(F)$.

For $i=1,2$, it is easy to see that $M(F)\subset P(F)\cap v_{i}^{-1}\bar{P}(F)v_i$. Together with the fact that $P(F)=M(F)H(F)$, we have $\bar{P}(F)v_iP(F)=\bar{P}(F)v_iH(F)$ for $i=1,2$. For $i=3$, $M(F)\cap v_{i}^{-1}\bar{P}(F)v_i\simeq B_{\GL_2}(F)\times \GU_{1,1}(F)$ where $B_{\GL_2}(F)$ is the upper triangular Borel subgroup of $\GL_2(E)$. Since the double coset
$$B_{\GL_2}(F)\times \GU_{1,1}(F)\back  \GL_2(E)\times \GU_{1,1}(F)/ \GU_{1,1}(F)^{\Delta}$$
contains two orbits that are represented by $I_2\times I_2$ and $\left(\begin{smallmatrix}1&0\\1&1\end{smallmatrix}\right)\times I_2$, we have $\bar{P}(F)v_3P(F)=\bar{P}(F)v_{31}H(F)\cup \bar{P}(F)v_{32}H(F)$ with
$$v_{31}=v_3,\;v_{32}=v_3\cdot diag(\begin{pmatrix}1&0\\1&1\end{pmatrix}, I_2, \begin{pmatrix}1&0\\-1&1\end{pmatrix}).$$
Similarly, since the double coset
$$B_{\GL_2}(F)\times \bar{B}_0(F)\back  \GL_2(E)\times \GU_{1,1}(F)/\GU_{1,1}(F)^{\Delta}$$
contains three elements that are represented by $I_2\times I_2,\;I_2\times \left(\begin{smallmatrix}0&1\\1&0\end{smallmatrix}\right)$ and $\left(\begin{smallmatrix}1&0\\1&1\end{smallmatrix}\right)\times I_2$ where $\bar{B}_{0}(F)$ is the lower triangular Borel subgroup of $\GU_{1,1}(F)$, we have
\begin{align*}
& \bar{P}(F)v_4P(F)=\bar{P}(F)v_{41}H(F)\cup \bar{P}(F)v_{42}H(F)\cup \bar{P}(F)v_{43}H(F)\\
&\bar{P}(F)v_5P(F)=\bar{P}(F)v_{51}H(F)\cup \bar{P}(F)v_{52}H(F)\cup \bar{P}(F)v_{53}H(F)
\end{align*}
with
\begin{align*}
 & v_{41}=v_4,\;v_{42}=v_4\cdot diag(\begin{pmatrix}1&1\\0&1\end{pmatrix}, I_2, \begin{pmatrix}1&-1\\0&1\end{pmatrix}),\; v_{43}=v_4\cdot w_{(124356)},\\
&v_{51}=v_5,\;v_{52}=v_5\cdot diag(\begin{pmatrix}1&0\\1&1\end{pmatrix}, I_2, \begin{pmatrix}1&0\\-1&1\end{pmatrix}),\; v_{53}=v_5\cdot w_{(124356)}.
\end{align*}
To summarize, the double coset $\bar{P}(F)\back G(F)/H(F)$ contains 10 orbits $\{\bar{P}(F)\gamma_iH(F)|\;1\leq i\leq 10\}$ with
\begin{align*}
&\gamma_1=I_6,\;\gamma_2=w_{(653421)},\;\gamma_3=w_{(623451)},\;\gamma_4=w_{(623451)}\cdot diag(\begin{pmatrix}1&0\\1&1\end{pmatrix}, I_2, \begin{pmatrix}1&0\\-1&1\end{pmatrix}),\\
&\gamma_5=w_{(351624)},\;\gamma_6=w_{(351624)}\cdot diag(\begin{pmatrix}1&1\\0&1\end{pmatrix}, I_2, \begin{pmatrix}1&-1\\0&1\end{pmatrix}),\; \gamma_{7}=w_{(351624)}\cdot w_{(124356)},\\
&\gamma_8=w_{(321654)},\;\gamma_9=w_{(321654)}\cdot diag(\begin{pmatrix}1&0\\1&1\end{pmatrix}, I_2, \begin{pmatrix}1&0\\-1&1 \end{pmatrix}), \;\gamma_{10}=w_{(321654)}\cdot w_{(124356)}.
\end{align*}

Now we are ready to prove Proposition \ref{A.2} for this case. For $2\leq i\leq 10$, we need to show that
\begin{equation}\label{A.5}
\Hom_{Q_i(F)} (\tau, (\delta_{Q_i}\delta_{\bar{Q}}^{-1/2})\otimes {}^{\gamma_i} \xi)=\{0\}
\end{equation}
with $Q_i(F)=\bar{P}(F)\cap \gamma_{i}H(F)\gamma_{i}^{-1}$. For $2\leq i\leq 9$, one can easily show that the character $(\delta_{Q_i}\delta_{\bar{Q}}^{-1/2})\otimes {}^{\gamma_i}\xi$ of $Q_i(F)$ is nontrivial on $\bar{U}(F)\cap \gamma_{i}H(F)\gamma_{i}^{-1}$, and then this proves \eqref{A.5} as the representation $\tau$ of $\bar{P}(F)$ is trivial on $\bar{U}(F)$. So it remains to consider the last orbit $\bar{P}(F)\gamma_{10}H(F)$. In this case, the character will be trivial on the unipotent part $\bar{U}(F)\cap \gamma_{10}H(F)\gamma_{10}^{-1}$. Hence we can get rid of this part. So it remains to consider the reductive part $(M(F),M(F)\cap \gamma_{10}H(F)\gamma_{10}^{-1})$. By an easy computation of the intersection and the character, it is enough for us to show that as a representation of $M(F)\simeq \GL_2(E) \times \GU_{1,1}(F)$, $\tau=\tau_1\otimes \tau_2$ is not $(M',\xi')$-distinguished where
\begin{align*}
M'(F)=&\{m(a,b,x,y)=\begin{pmatrix} b&0\\x&b\end{pmatrix}\times \begin{pmatrix}a&y\\0&b\end{pmatrix}\colon \\
&  a,b\in E^{\times}, x,y\in E \text{ with } \frac{a}{b}\in F^{\times}\;\text{and} \; \frac{y}{a}\in \sqrt{\alpha}F \}\\
\end{align*}
and the character $\xi'$ is defined to be
$$\xi'(m(a,b,x,y))=|\frac{a}{b}|^2 \psi(\tr_{E/F}(\frac{x}{b})).$$
Here the $|\frac{a}{b}|^2$-part comes from the modular character and the $\psi(\tr_{E/F}(\frac{x}{b}))$-part comes from the character $\xi$. After we modulo the center, it is enough to show that $\tau$ is not $(M'',\xi'')$-distinguished where
$$M''(F)=\{m'(a,x,y)=\begin{pmatrix} 1&0\\x&1\end{pmatrix}\times \begin{pmatrix}a&y\\0&1\end{pmatrix}\colon a\in F^{\times}, x\in E,y\in \sqrt{\alpha}F\}$$
and $\xi''(m'(a,x,y))=|a|^2 \psi(\tr_{E/F}(x))$.

Let $B_0=T_0N_0$ be the upper triangular Borel subgroup of $\GU_{1,1}(F)$ and let $J_{N_0}(\tau_2)$ be the Jacquet module of $\tau_2$ with respect to $N_0(F)$. Then in order to show $\tau=\tau_1\otimes \tau_2$ is not $(M',\xi')$-distinguished, it is enough to show that as a representation of $\GL_2(E)\times T_0(F)$, the representation $\tau_1\otimes J_{N_0}(\tau_2)$ is not $(M_0(F),\xi_0)$-distinguished where
\begin{align*}
 & M_0(F)=\{m_0(a,x)=\begin{pmatrix} 1&0\\x&1\end{pmatrix}\times \begin{pmatrix}a&0\\0&1\end{pmatrix}\colon a\in F^{\times}, x\in E\},\\
 &\xi_0(m_0(a,x))=|a|^2 \psi(\tr_{E/F}(x)).
\end{align*}
If $\tau_2$ is supercuspidal, $J_{N_0}(\tau_2)=0$ and hence $\tau_1\otimes J_{N_0}(\tau_2)$ is not $(M_0(F),\xi_0)$-distinguished. When $\tau_2$ is a discrete series, $J_{N_0}(\tau_2)$ is a one dimension representation of $T_0(F)$ with
$$|J_{N_0}(\tau_1)(\begin{pmatrix}a&0\\0&b\end{pmatrix})|=|\frac{a}{b}|,\; a,b\in E^{\times},\frac{a}{b}\in F^{\times}.$$
This implies that $\tau_1\otimes J_{N_0}(\tau_2)$ is not $(M_0(F),\xi_0)$-distinguished. When $\tau_2$ is a tempered representation but not a discrete series, $J_{N_0}(\tau_2)=\eta_1\oplus \eta_2$ where $\eta_i$ are characters of $T_0(F)$ with
$$|\eta_i(\begin{pmatrix}a&0\\0&b\end{pmatrix})|=|\frac{a}{b}|^{1/2},\; a,b\in E^{\times},\frac{a}{b}\in F^{\times}.$$
This implies that $\tau_1\otimes J_{N_0}(\tau_2)$ is not $(M_0(F),\xi_0)$-distinguished. To summarize, we have proved Proposition \ref{A.2} when $M_{Q}(F)$ is isomorphic to $\GL_2(E)\times \GU_{1,1}(F)$.
\\

Then we consider the case when $M_{Q}(F)\simeq \GL_3(E)\times \GL_1(F)$. In this case, we may choose $\bar{Q}(F)$ to be the parabolic subgroup of $\GU_{3,3}(F)$ containing the lower Borel subgroup such that its Levi part is isomorphic to $\GL_3(F)\times \GL_1(F)$. By a similar argument as in the previous case, we can show that the double coset $\bar{Q}(F)\back G(F)/H(F)$ contains 5 orbits $\{\bar{Q}(F)\gamma_iH(F)|\;1\leq i\leq 5\}$ with
\begin{align*}
&\gamma_1=I_6,\;\gamma_2=w_{(654321)},\;\gamma_3=w_{(623451)},\\
&\gamma_4=w_{(623451)}\cdot diag(\begin{pmatrix}1&0\\1&1\end{pmatrix}, I_2, \begin{pmatrix}1&0\\-1&1 \end{pmatrix}),~
\gamma_{5}=w_{(623451)}\cdot w_{(124356)}.
\end{align*}
For $2\leq i\leq 5$, we need to show that
\begin{equation}\label{A.6}
\Hom_{Q_i(F)} (\tau, (\delta_{Q_i}\delta_{\bar{Q}}^{-1/2})\otimes {}^{\gamma_i} \xi)=\{0\}
\end{equation}
with $Q_i(F)=\bar{Q}(F)\cap \gamma_{i}H(F)\gamma_{i}^{-1}$. The argument is similar to the previous case. For $2\leq i\leq 4$, one can easily show that the character $(\delta_{Q_i}\delta_{\bar{Q}}^{-1/2})\otimes {}^{\gamma_i}\xi$ of $Q_i(F)$ is nontrivial on $\bar{U}_Q(F)\cap \gamma_{i}H(F)\gamma_{i}^{-1}$, and then this proves \eqref{A.6} as the representation $\tau$ is trivial on $\bar{U}_Q(F)$. So it remains to consider the last orbit $\bar{Q}(F)\gamma_{5}H(F)$. In this case, the character will be trivial on the unipotent part $\bar{U}_Q(F)\cap \gamma_{5}H(F)\gamma_{5}^{-1}$. Hence we can get rid of this part. So it remains to consider the reductive part $(M_Q(F),M_Q(F)\cap \gamma_{5}H(F)\gamma_{5}^{-1})$. By an easy computation of the intersection and the character, it is enough for us to show that as representation of $M(F)\simeq \GL_3(E)\times \GL_1(F)$, $\tau=\tau_1\otimes \tau_2$ is not $(M',\xi')$-distinguished where
$$M'(F)=\{m(a,b,x,y,z)=\begin{pmatrix}a\end{pmatrix} \times \begin{pmatrix} b&x&z\\0&b&y\\0&0&b\end{pmatrix} \colon a\in F^{\times},b\in E^{\times}, x,y,z\in E \}$$
and the character $\xi'$ is defined by
$$\xi'(m(a,b,x,y,z))=|\frac{b\bar{b}}{a}|^{1/2} \psi(\tr_{E/F}(\frac{x}{b}+\frac{y}{b})).$$
Here the $|\frac{b\bar{b}}{a}|^{1/2}$-part comes from the modular character and the $\psi(\tr_{E/F}(\frac{x}{b}+\frac{y}{b}))$-part comes from the character $\xi$. But this is trivial since $\tau$ is a tempered (and hence unitary) representation. This proves Proposition \ref{A.2} when $M_{Q}(F)$ is isomorphic to $\GL_3(E)\times \GL_1(F)$.

Finally we consider the case when $M_{Q}(F)\simeq \GL_1(E)\times \GU_{2,2}(F)$. In this case, let $Q'(F)$ be the parabolic subgroup of $\GU_{3,3}(F)$ containing the lower Borel subgroup such that its Levi part is isomorphic to $\GL_1(E)\times \GU_{2,2}(F)$. Then we can choose $\bar{Q}(F)$ to be $\delta Q'(F)\delta^{-1}$ with
$$\delta=diag(\begin{pmatrix}1&1\\0&1\end{pmatrix}, I_2, \begin{pmatrix}1&-1\\0&1\end{pmatrix}).$$
By a similar argument as in the previous cases, we can show that the double coset $\bar{Q}(F)\back G(F)/H(F)$ contains 5 orbits $\{\bar{Q}(F)\gamma_iH(F)|\;1\leq i\leq 5\}$ with
$$\gamma_1=I_6,\;\gamma_2=\delta w_{(623451)},\;\gamma_3=\delta w_{(623451)}\delta ,\;\gamma_4=\delta w_{(321654)},\;\gamma_5=\delta.$$
For $2\leq i\leq 5$, we need to show that
\begin{equation}\label{A.7}
\Hom_{Q_i(F)} (\tau, (\delta_{Q_i}\delta_{\bar{Q}}^{-1/2})\otimes {}^{\gamma_i} \xi)=\{0\}.
\end{equation}
The argument is similar to the previous cases. For $2\leq i\leq 4$, one can easily show that the character $(\delta_{Q_i}\delta_{\bar{Q}}^{-1/2})\otimes {}^{\gamma_i}\xi$ of $Q_i(F)$ is nontrivial on $\bar{U}_Q(F)\cap \gamma_{i}H(F)\gamma_{i}^{-1}$, and this proves \eqref{A.7} as the representation $\tau$ is trivial on $\bar{U}_Q(F)$. So it remains to consider the last orbit $\bar{Q}(F)\gamma_{5}H(F)$. In this case, the character will be trivial on the unipotent part $\bar{U}_Q(F)\cap \gamma_{5}H(F)\gamma_{5}^{-1}$. Hence we can get rid of this part. So it remains to consider the reductive part $(M_Q(F),M_Q(F)\cap \gamma_{5}H(F)\gamma_{5}^{-1})$. By an easy computation of the intersection and the character, it is enough for us to show that as a representation of $M(F)\simeq \GL_1(E)\times \GU_{2,2}(F)$, $\tau=\tau_1\otimes \tau_2$ is not $(M',\xi')$-distinguished where
\begin{align*}
M'(F)=&\{m(a,b,x,X)=\begin{pmatrix}a\end{pmatrix}\times \begin{pmatrix}b&x&0&0\\0&b&0&0\\0&0&a&-\frac{a\bar{x}}{\bar{b}} \\0&0&0&a\end{pmatrix} \begin{pmatrix} I_2&X \\0&I_2\end{pmatrix}\colon \\
& a,b\in E^{\times},x\in E,X\in Mat_{2\times 2}(E)\;\text{with}\;\frac{a}{b}\in F^{\times},w_2\bar{X}w_2+{}^tX=0 \},
\end{align*}
and the character $\xi'$ is defined by
$$\xi'(m(a,b,x,X))=|\frac{b}{a}|^{3/2} \psi(\tr_{E/F}(\frac{x}{b})).$$
Here the $|\frac{b}{a}|^{3/2}$-part comes from the modular character and the $\psi(\tr_{E/F}(\frac{x}{b}))$-part comes from the character $\xi$. Let $Q_0=L_0N_0$ be the standard parabolic subgroup of $\GU_{2,2}(F)$ with $L_0(F)\simeq \GL_2(E)\times \GL_1(F)$ and let $J_{N_0}(\tau_2)$ be the Jacquet model of $\tau_2$ which is a representation of $L_0(F)=\{diag(\lambda g,w_2{}^t\bar{g}^{-1}w_2)\colon g\in \GL_2(E),\lambda\in F^{\times}\}$. Then in order to show $\tau=\tau_1\otimes \tau_2$ is not $(M',\xi')$-distinguished, it is enough to show that as a representation of $\GL_1(E)\times L_0(F)$, the representation $\tau_1\otimes J_{N_0}(\tau_2)$ is not $(M_0(F),\xi_0)$-distinguished where
\begin{align*}
  M_0(F)=&\{m_0(a,b,x)=\begin{pmatrix}a\end{pmatrix}\times diag(\begin{pmatrix}b&x\\0&b\end{pmatrix},\begin{pmatrix}a&-\frac{a\bar{x}}{\bar{b}}\\0&a \end{pmatrix}) \colon \\
 &a,b\in E^{\times},x\in E \;\text{with}\; \frac{a}{b}\in F^{\times}\},\\
\xi_0(m_0(a,b,x))=&|\frac{b}{a}|^{3/2} \psi(\tr_{E/F}(\frac{x}{b})).
\end{align*}

We decompose $J_{N_0}(\tau_2)$ as $\oplus_{i=1}^{k}\tau_{2,i}$ based on the central character (i.e. the center of $L_0(F)$ acts by scalar on $\tau_{2,i}$ and the central characters of $\tau_{2,i}$ and $\tau_{2,j}$ are different for any $i\neq j$). We use $\chi_i$ to denote the central character of $\tau_{2,i}$. Since $\tau_2$ is tempered, by Proposition III.2.2 of \cite{W03}, all the characters $\chi_i$ are ``positive" than the square root of the modular character. To be specific, for all $1\leq i\leq k$, we have
$$|\chi_i(diag(\begin{pmatrix}b&0\\0&b\end{pmatrix}\times \begin{pmatrix}\bar{b}^{-1}&0\\0&\bar{b}^{-1}\end{pmatrix}))|=|b\bar{b}|^{s_i}, \;b\in E^{\times}$$
for some $s_i\in \BR$ with $s_i\geq 2$. On the mean time, we have
$$\xi_0(\begin{pmatrix}\bar{b}^{-1}\end{pmatrix}\times diag(\begin{pmatrix}b&0\\0&b\end{pmatrix}\times \begin{pmatrix}\bar{b}^{-1}&0\\0&\bar{b}^{-1} \end{pmatrix}))=|b\bar{b}|^{3/2}.$$
This implies that the representation $\tau_1\otimes J_{N_0}(\tau_2)$ is not $(M_0(F),\xi_0)$-distinguished. This proves Proposition \ref{A.2} when $M_{Q}(F)$ is isomorphic to $\GL_1(E)\times \GU_{2,2}(F)$. Now the proof of Proposition \ref{A.2} is complete.

\end{document}